\documentclass{mrlart7}
  \def\volno{18}
  \def\yrno{2011}
  \def\issueno{00}
  \def\lpageno{10029} 
\usepackage{amsmath,amssymb,amsthm,amsfonts,enumerate,array,color,lscape,fancyhdr,layout,pst-all}
\usepackage[english]{babel}
\usepackage[latin1]{inputenc}
\usepackage{young}

\usepackage[hcentermath]{youngtab}
\usepackage{graphicx}
\usepackage{float}
\usepackage{pstricks}
\usepackage[all]{xy}

\newcounter{numberofremark}

\newcommand\nothing[1]{}
\usepackage[active]{srcltx}
\usepackage[hyperindex,bookmarksnumbered,plainpages]{hyperref}

\newcommand{\dcl}{\DeclareMathOperator}
\dcl\cdet{cdet} \dcl\Sp{Specm} \dcl\depth{depth} \dcl\im{Im} \dcl\Span{span} \dcl\Ker{Ker} \dcl\Specm{Specm}
\dcl\Supp{Supp} \dcl\codim{codim} \dcl\Y{Y} \dcl\gl{\mathfrak{gl}}    \dcl\U{U} \dcl\T{T}
\dcl\qdet{qdet} \dcl\sgn{sgn} \dcl\gr{gr} \dcl\diag{diag}
\dcl\g{\mathfrak{g}} \dcl\C{\mathbb C} \dcl\dd{{\mathrm d}}
\dcl\Ind{Ind}\dcl\Hom{Hom}

\newcommand\La{{\Lambda}}
\newcommand\la{{\lambda}}

\newcommand\bm{{\mathbf m}}

\newlength\yStones
\newlength\xStones
\newlength\xxStones

\makeatletter
\def\Stones{\pst@object{Stones}}
\def\Stones@i#1{%
  \pst@killglue%
  \begingroup%
  \use@par%
  \setlength\xxStones{\xStones}%
  \expandafter\Stones@ii#1,,\@nil
  \endgroup
  \global\addtolength\xStones{0.6cm}%
  \global\addtolength\yStones{-7.5mm}}%
\def\Stones@ii#1,#2,#3\@nil{%
  \rput(\xxStones,\yStones){%
    \psframebox[framesep=0]{%
      \parbox[c][6mm][c]{11mm}{\makebox[11mm]{$#1$}}}}%
  \addtolength\xxStones{1.2cm}%
  \ifx\relax#2\relax\else\Stones@ii#2,#3\@nil\fi}
\makeatother

\def\Stone#1{\fbox{\makebox[10mm]{\strut#1}}\kern2pt}

\newtheorem{theorem}{Theorem}[section]
\newtheorem{lemma}[theorem]{Lemma}
\newtheorem{corollary}[theorem]{Corollary}

\newtheorem{proposition}[theorem]{Proposition}
\newtheorem{example}[theorem]{Example}
\newtheorem{remark}[theorem]{Remark}

\newtheorem{definition}[theorem]{Definition}

\begin{document}
\title{Gelfand-Tsetlin modules for $\mathfrak{gl}(m|n)$}
\author{Vyacheslav Futorny}
\address{Instituto de Matem\'atica e Estat\'istica, Universidade de S\~ao Paulo,  S\~ao Paulo, Brazil, and  International Center for Mathematics, SUSTech, Shenzhen, China}\email{futorny@ime.usp.br}
\author{Vera Serganova}
\address{Berkeley}\email{serganov@math.berkeley.edu }
\author{Jian Zhang}
\address{Institute of Mathematics, Academia Sinica, Taipei, Taiwan 10617} \email{jzhang@gate.sinica.edu.tw}

\begin{abstract}
We address the  problem of classifying irreducible Gelfand-Tsetlin modules for $\gl(m|n)$ and show that it reduces to the 
classification of Gelfand-Tsetlin modules for the even part. We also give an explicit tableaux construction and the irreducibility criterion for the class of quasi typical and quasi covariant Gelfand-Tsetlin modules 
which includes all essentially typical and covariant tensor  finite dimensional modules. In the quasi typical case 
new irreducible representations  are infinite dimensional
$\gl(m|n)$-modules which are isomorphic to the parabolically induced (Kac) modules.

\end{abstract}

\maketitle
\section{Introduction}
Throughout the paper the field is assumed to be the field of complex numbers. Denote by 
Let $\g$ be the Lie superalgebra  $\gl(m|n)$ which 
 consists of  $(m+n)\times (m+n) $ complex matrices of the block form
\begin{equation}
g=\left(\begin{array}{ll}
A & B\\
C & D
\end{array}\right)
\end{equation}
where $A,B,C$ and $D $  are respectively ${m\times m}, {m\times n}, {n\times m}$ and ${n\times n}$ matrices.
The even subalgebra $(\g_{m,n})_{\bar{0}}\simeq \gl(m)\oplus \gl(n)$ consists matrices with $B=0$ and $C=0$; the odd subspace $(\g_{m,n})_{\bar{1}}$ consists those with $A=0$ and $D=0$. The Lie superalgebra is then defined by the bracket $[x,y]=xy-(-1)^{|x||y|}yx$, where $x$ and $y$ are homogeneous elements.

The Cartan subalgebra  $\mathfrak h$ of $\mathfrak{gl}(m|n)$ is the subalgebra of diagonal matrices.  We denote by $U(\g)$  the universal enveloping algebra 
of $\g$.  

A chain of subalgebras of $\g$, 
$$\g=\g^{1}\supset \g^{2}\supset \ldots \supset \g^{m+n},$$
such that  $\g^k$ is isomorphic to $\gl(p|q)$ with $p+q=m+n-k+1 $ is called a \emph{complete flag} in $\g$  if 
$\mathfrak h^k:=\mathfrak h\cap\g^k$ is a Cartan subalgebra of $\g^k$. A complete flag  induces a chain of Cartan subalgebras 
$$\mathfrak h=\mathfrak h^{1}\supset \mathfrak h^{2}\supset \ldots \supset \mathfrak h^{m+n}.$$

Every complete flag $\mathcal C$ in $\g$ defines the commutative subalgebra $\Gamma_{\mathcal C}$ in $U(\g)$ generated by the centers of the members of the chain. We will say that $\Gamma_{\mathcal C}$ is the \emph{Gelfand-Tsetlin subalgebra} of $U(\g)$
associated with the flag $\mathcal C$.

Let  $\Gamma=\Gamma_{\mathcal C}$ be a Gelfand-Tsetlin subalgebra of $U(\g)$. 
A finitely generated module $M$ over $\g$ is called a {\em
Gelfand-Tsetlin module\/} (with respect to $\Gamma$) if \begin{equation*}
M=\underset{{\bm} \in \Specm {\Gamma}}{\bigoplus}M({\bm}) \end{equation*} as
a $\Gamma$-module, where \begin{equation*} M({\bm}) \ = \ \{ x\in M\ | \ {\bm}^k x
=0\quad \text{for some}\quad k\geqslant 0\} \end{equation*}
 and $\Specm \Gamma$
denotes the set of maximal ideals of $\Gamma$.  Identifying maximal ideals of $\bm\in \Specm \Gamma$ with kernels of
characters $\chi\in \hat{\Gamma}$,  Gelfand-Tsetlin module $M$  (with respect to $\Gamma$) will have the following form:
\begin{equation*}
M=\underset{\chi\in \hat{\Gamma}}{\bigoplus}M({\chi}), \end{equation*} where \begin{equation*} M({\chi}) \ = \ \{ x\in M\ | \  (z-\chi(z))^k x
=0\quad \text{for some}\quad k\geqslant 0  \quad \text{and all} \quad  z\in \Gamma   \}. \end{equation*}

Recall that $\g=\gl(m|n)$ has a consistent $\mathbb Z$-grading $\g=\g_{-1}\oplus\g_0\oplus\g_1$.
There is a bijection between irreducible $\g$-supermodules and irreducible $\g_0$-supermodules \cite{CM}.
The purpose of our paper is to address the classification problem of irreducible Gelfand-Tsetlin modules. We show that this problem reduces to the 
classification of Gelfand-Tsetlin modules for the even part of $\g$ with respect to the Gelfand-Tsetlin subalgebra $\phi(\Gamma)$, where
 $\phi: U(\g)\to U(\g_0)$ is the projection with the kernel $\g_{-1}U(\g)+U(\g)\g_1$. 
 Our first result is summarized in the following theorem

\begin{theorem}\label{theorem-main}
Let $\Gamma$ be a Gelfand-Tsetlin subalgebra of $U(\g)$,
$\Gamma_0=\phi(\Gamma)$.  
\begin{itemize} 
\item[(i)] A $\g$-module $M$ is locally finite over $\Gamma$ if and only if it is locally finite over $\Gamma_0$.

\item[(ii)]  Let  $M$ be an irreducible  Gelfand-Tsetlin (with respect to $\Gamma$) $\g$-module.
Then
$M$ is a unique simple quotient of the induced (Kac) module $$K(N):=U(\g)\otimes_{U(\g_0\oplus\g_1)}N$$ 
for some irreducible Gelfand-Tsetlin (with respect to $\Gamma_0$)  $\g_0$-module   $N$ (we assume $\g_1N=0$).

\item[(iii)] For any $\chi_0\in \hat{\Gamma}_0$ there exists only finitely many
non-isomorphic irreducible Gelfand-Tsetlin (with respect to
$\Gamma$) $\g$-modules $M$ with $M(\chi_0)\neq 0$.

\item[(iv)] If $\Gamma'$ is another Gelfand-Tsetlin subalgebra of $\g$
such that $\phi(\Gamma')=\Gamma_0$ then the categories of
Gelfand-Tsetlin  $\g$ -modules for $\Gamma$ and $\Gamma'$
coincide.

\item[(v)] The subcategory of Gelfand-Tsetlin (with respect to $\Gamma$) $\g$-modules with a fixed typical central character is
equivalent to  the subcategory of Gelfand-Tsetlin (with respect to $\Gamma_0$) $\g_0$-modules with a suitable central
character.
\end{itemize}
\end{theorem}

A classical paper of  Gelfand  and  Tsetlin \cite{GT} gives a construction of simple finite dimensional modules for $\g_0$ with explicit basis consisting of Gelfand-Tsetlin tableaux and the action of the Lie algebra. Infinite dimensional Gelfand-Tsetlin modules were studied extensively and many remarkable connections and applications were discovered (cf. \cite{DFO}, \cite{EMV}, \cite{FGRZ}, \cite{FO}, \cite{H}, \cite{KW1},  \cite{KW1}, \cite{KTWWY},   \cite{O}, \cite{W} and references therein).

In \cite{FRZ} a new combinatorial method of constructing of Gelfand-Tsetlin modules was developed. It is based on  the work of  Gelfand and Graev \cite{GG} and of  Lemire and Patera \cite{LP} who
 initiated a  study of formal continuations
of both the labelling tableaux and the classical Gelfand-Tsetlin formulas. 
Imposing certain modified conditions  on the entries of a  tableau  new infinite dimensional
irreducible modules can be constructed explicitly 
with the Lie algebra action given by the classical Gelfand-Tsetlin formulas.

Analogs of the Gelfand--Tsetlin bases for a certain class of finte dimensional \emph{essentially typical} representations of the Lie superalgebra $\gl(m|n)$ were constructed in \cite{Palev1989a}, 
\cite{Palev1989b}.   
Essentially typical representations is a class of typical representations that have consecutive restrictions   to the subalgebras of the chain 
$\gl(m|1)\subset \ldots \subset \gl(m|n)$  completely reducible.
 This allows  to apply the classical  Gelfand-Tsetlin formulas   to obtain explicit construction of all essentially typical representations. A basis is given by special Gelfand-Tsetlin tableaux.

Also alternative explicit construction of finite dimensional irreducible \emph{covariant tensor} modules over $\gl(m|n)$  was given in \cite{M2}, \cite{Stoilova2010}. Highest weight  
of such representation is a $m+n$-tuple $\la=(\la_{1},\la_{2},\ \ldots,\la_{m+n})$
of non-negative integers  
such that
\begin{equation}
\la_i-\la_{i+1}\in \mathrm{Z}_{\geq 0}, \text{ for }i=1, \ldots,m+n-1,  \quad i\neq m
\end{equation}
and
\begin{equation}
\la_m\geq\#\{i:\la_{i}>0,\ m+1\leq i\leq m+n\}.
\end{equation}

We generalize  the constructions of essentially typical and covariant tensor representations using the combinatorial approach developed in \cite{FRZ}.
It allows us to obtain large families of  \emph{quasi typical}  
and \emph{quasi covariant}
Gelfand-Tsetlin modules respectively,  together with their 
explicit tableau realization.  In particular, these families contain all finite dimensional essentially typical and covariant tensor representations. We give necessary and sufficient    conditions of the irreducibility  of constructed modules.  These results are subjects of Theorem \ref{admissible}, Theorem \ref{irreducibility}, Theorem \ref{Kac-module} and Theorem \ref{quasi covariant}. We summarize the statements in the following theorem.

\begin{theorem}\label{theorem-main2}
Let $\mathcal{C}$ be an admissible (respectively, covariant admissible) set of relations. 
\begin{itemize}
\item[(i)] The
 formulas \eqref{action h}-\eqref{action f} define a quasi typical (respectively, quasi covariant) $\gl(m|n)$-module structure on $V_{\mathcal C}([l^0])$ for any 
tableau $[l^0]$ satisfying $\mathcal{C}$ (respectively, for any
$\mathcal{C}$-covariant tableau
$[l^0]$). 
\item[(ii)] The module
$V_{\mathcal C}([l^0])$ is irreducible if and only if
$\mathcal C$ is the
maximal set of relations satisfied by $[l^0]$ and $l^0_{m+n,i}\neq l^0_{m+n, j}, 1\leq i\leq m<  j\leq m+n$ in the quasi typical case.
\item[(iii)] The module
$V_{\mathcal C}([l^0])$ is irreducible if and only if
$\mathcal C$ is the
maximal set of relations satisfied by $[l^0]$ in the quasi covariant case.
\item[(iv)] If 
$V_{\mathcal C}([l^0])$ is irreducible  quasi typical module then it is isomorphic to the induced module
$K(V_{\mathcal{C}}([l^0])^{\g_1})$.
\end{itemize}

\end{theorem}

\

\section{Preliminaries}
\subsection{Weight modules}
A $\mathbb Z_2$-graded vector space $\mathfrak g=\mathfrak g_{\bar{0}}\oplus \mathfrak g_{\bar{1}}$ with even bracket
$[\bullet,\bullet]:\mathfrak g\otimes \mathfrak g\to \mathfrak g$ is a Lie superalgebra iff the following conditions hold
$$[a,b]=-(-1)^{p(a)p(b)}[b,a];$$
$$[a,[b,c]]=[[a,b],c]+(-1)^{p(a)p(b)}[a,[b,c]].$$

Let $\g$ be a finite dimensional Lie superalgebra equipped
with a grading $\g=\g_{-1}\oplus \g_0 \oplus \g_1$, where $\g_{\pm 1}$
are odd abelian subalgebras and $\g_0=g_{\bar{0}}$ is even. Assume that $\g_0$, $\g_{\pm 1}$
are semisimple $\g_0$-modules and $\g_0$ contains an element $d$ such that $[d,x]=ix$ for any
$x\in\g_i$.

Let $\mathfrak h$ denote a Cartan subalgebra of $\mathfrak g_0$
 and assume that  it coincides with its centralizer in $\mathfrak g$.
Then $\mathfrak g$ has a root decomposition
$$\mathfrak g=\mathfrak h\oplus\bigoplus_{\alpha\in\Delta\subset\mathfrak h^*}\mathfrak g_{\alpha},$$
where $\mathfrak g_{\alpha}=\{x\in\mathfrak g | [h,x]=\alpha(h)x\}.$

We fix the sets of positive and
negative roots  $\Delta=\Delta^+\cup\Delta^-$ and a corresponding triangular
decomposition
$$\mathfrak g=\mathfrak n^-\oplus\mathfrak h\oplus\mathfrak n^+, $$ where
$\mathfrak n^{\pm}=\bigoplus\mathfrak g_\alpha$, $ \alpha\in\Delta^{\pm}$.

A $\mathfrak g$-module
(respectively, $\mathfrak g_0$-module) $M$ is called a \emph{weight} module (with respect to $\mathfrak h$) if
$$M=\oplus_{\lambda\in \mathfrak h^*} M_{\lambda},$$
where $M_{\lambda}=\{v\in M | hv=\lambda(h)v  \quad \text{for all}\quad h\in \mathfrak h\}$.

Let $M$ be a weight module over $\g_0$. Denote by $K(M)$ the
induced module: $$K(M)=U(\g)\otimes_{U(\g_0\oplus \g_1)}M,$$
where $\g_1 M=0$. If $M$ is irreducible then $K(M)$ has a
unique irreducible quotient which we will denote  $L(M)$.


Hence, $\bf K$, which sends $M$ to $K(M)$, defines a functor from the category of weight  $\g_0$-modules to the category of weight $\g$-modules.
On the other hand, for a weight $\g$-module $V$ set $R(V)=V^{\g_1}$ which is a weight $\g_0$-module. Thus we obtain the functor $\bf R$, which sends $V$ to $R(V)$,
from the category of weight  $\g$-modules to the category of weight  $\g_0$-modules.

The following statement is an immediate consequence of the Frobenius reciprocity for induced modules.
\begin{proposition}\label{prop-adjoint}
 The pair $(\bf K, \bf R)$ is adjoint.
\end{proposition}

\

Important examples of weight modules are Verma modules and their irreducible quotients.
Let $\mathfrak b=\mathfrak h\oplus\mathfrak n^+$, $\lambda\in
\mathfrak h^*$. A Verma module with highest weight $\lambda$ is the following induced module
$$M(\lambda)=U(\mathfrak g)\otimes_{U(\mathfrak b)}Cv_{\lambda},$$
where $\mathfrak n^+v=0$ and $hv=\lambda(h)v$ for all $h\in \mathfrak h$.
It has a unique irreducible quotient $L(\lambda)$.

For a dominant weight $\lambda$ denote by $L_0(\lambda)$ the unique irreducible $\g_0$-module with highest weight 
$\lambda$. Then 
$K(\lambda):=K(L_0(\lambda))$ is the indecomposable \emph{Kac module} with highest weight $\lambda$. 

 Let
$\rho=\frac12\sum_{\alpha\in\Delta+}(-1)^{p(\alpha)}\alpha$ and $\lambda\in \mathfrak h^*$.
 A dominant  weight $\lambda$ is \emph{typical} if  
 $(\lambda+\rho,\alpha)\neq 0$ for any odd positive root $\alpha$. In the case of dominant typical $\lambda$, $K(\lambda):=L(\lambda)$.

We have $U(\g)=(\g_{-1}U(\g)+U(\g)\g_1)\oplus U(\g_0)$. Then the projection onto $U(\g_0)$ defines an analogue of the \emph{Harish-Chandra homomorphism}
 $\phi: U(\g)^d \rightarrow U(\g_{\bar{0}})$
 where $U(\g)^d=\{y\in U(\g)\mid [d,y]=0\}$.

Let $Z(\mathfrak g)$ denote the center of universal enveloping algebra $U(\mathfrak g)$ \cite{BZV}, \cite{K2}, \cite{Se}.
Then $\phi(Z(\mathfrak g))\subset Z(\mathfrak g_0)$, where $Z(\mathfrak g_0)$ is the center of the universal enveloping algebra $U(\mathfrak g_0)$. If $\phi_0$ is the classical Harish-Chandra homomorphism for $\g_0$, then  $\phi_0\circ \phi:U(\mathfrak g)\to U(\mathfrak h)=\text{Pol}(\mathfrak h^*)$. 

\

\subsection{The Lie superalgebra $\mathfrak{gl}(m|n)$}
The underlying vector space of the Lie superalgebra $\g_{m,n}=\mathfrak{gl}(m|n)$
is spanned by
the standard basis elements $E_{ij}, 1\leq i,j \leq m+n$. The $\mathrm{Z}_{2}$-grading on $\gl(m|n)$ is defined by $E_{ij}\mapsto \bar{i}+\bar{j}$, where $\bar{i}$ is an element of $\mathrm{Z}_{2}$ which equals $0$ if $i\leq m$
and $1$ if $i>m$. The commutation relations in this basis are given by
\begin{equation}
[E_{ij},\ E_{kl}]=\delta_{kj}E_{il}-\delta_{il}E_{kj}(-1)^{(\bar{i}+\bar{j})(\bar{k}+\bar{l})}.
\end{equation}
%
Alternatively, $\mathfrak{g}_{m,n}$ can be defined  as the quotient of the free Lie superalgebra $\hat{\mathfrak{g}}$ over $\mathbb{C}$ generated by  $e_{i}, f_{i}, (1 \leq i\leq  m+n-1)$, $h_{j}, (1\leq j\leq m+n)$,
subject to the following relations (unless stated otherwise, the indices below run over all possible values):
\begin{align}
[h_{i},h_{j}]=0;\\
[h_{i},e_{j}]=(\delta_{ij}-\delta_{i,j+1})e_{j};\\
[h_{i},f_{j}]=-(\delta_{ij}-\delta_{i,j+1})f_{j};\\
[e_{i},f_{j}]=0  \text{ if }  i\neq j;\\
[e_{i},f_{i}]=h_{i}-h_{i+1} \text{ if } i\neq m;\\
[e_{m},f_{m}]=h_{m}+h_{m+1};
\\
[e_{i},e_{j}]=[f_{i},f_{j}]=0 \text{ if } |i-j|> 1;\\
[e_m,e_m]=[f_m,f_m]=0; \\
[e_i,[e_i,e_{i\pm 1}]]=[f_i,[f_i,f_{i\pm 1}]]=0, \text{ for } i\neq m;\\
[e_m[e_{m\pm1},[e_m,e_{m\mp1}]]]=[f_m[f_{m\pm 1},[f_m,f_{m\mp1}]]]=0;
\end{align}
The isomorphism between two algebras is defined by
$e_i\mapsto E_{i,i+1}, f_i\mapsto E_{i+1,i},1\leq i\leq m+n-1 $ and $h_i\mapsto E_{ii},1\leq i\leq m+n$. 
We define the parity function on $\{1,\dots,m+n\}$ by setting 
$$\bar {j}=\begin{cases} 0\,\,j\leq m\\1\,\,j\geq m+1\end{cases}.$$

Let $\hat{E}$  be the $(m+n)\times(m+n)$  matrix with coefficients in $U(\g)$ whose $ij$-th entry equals $\hat{E}_{ij}= (-1)^{\bar{j}}E_{ij}$.
The {\it quantum Berezinian} is an element in $U(\g)[[t]]$ defined by
\begin{equation}
\begin{split}
B(t)=\sum_{\sigma\in S_{m}} sgn  (\sigma)(1+t\hat{E})_{\sigma(1),1}\cdots(1+t(\hat{E}-m+1))_{\sigma(m),m}
\\
 \times\sum_{\tau\in S_{n}}
 sgn(\tau)(1+t(\hat{E}-m+1))_{m+1,m+\tau(1)}^{-1}\cdots(1+t(\hat{E}-m+n))_{m+n,m+\tau(n)}^{-1}.
\end{split}
\end{equation}
The coefficients of quantum Berezinian  generate the center of $U(\gl_{m,n})$, see \cite{Nazarov1991}, \cite{M1}.

%

Fix the Borel subalgebra $\g$ containing $\g_1$.
Let  $x_{i}=\lambda_{i}-i+1$ for $i=1, \ldots, m,$ 
$x_{j}=-\lambda_{j}+j-2m$ for $j=m+1, \ldots, m+n.$
Then $B(t)$  acts on $L(\lambda)$ by the scalar
\begin{equation}\label{hc}
\frac{(1+tx_{1})\cdots(1+tx_{m})}
{(1+tx_{m+1})\cdots(1+tx_{m+n})}.
\end{equation}

%
The even and the odd roots  in the standard basis are respectively the following:
$$\begin{array}{c}\Delta_0=\{(\varepsilon_i-\varepsilon_j) |i,j\leq m\}\cup
\{(\delta_i-\delta_j) |i,j\leq n\},  \\    \Delta_1=\{\pm(\varepsilon_i-\delta_j) |i\leq m,j\leq n\},\end{array}$$
 where $$(\varepsilon_i,\delta_j)=0,\,(\varepsilon_i,\varepsilon_j)=\delta_{ij},\,(\delta_i,\delta_j)=-\delta_{ij}.$$

\section{Gelfand-Tsetlin subalgebras}

Set $\g_{k,l}=\mathfrak{gl}(k|l)$ where $k\geq 0$ and $l\geq 0$. We call a
\emph{Gelfand-Tsetlin chain} any chain of subalgebras $\g_{k,l}$
in $\g_{m|n}$ such that on every step we have either
$\g_{k,l}\supset \g_{k-1,l}$ or $\g_{k,l}\supset \g_{k,l-1}$.

Fix a Gelfand-Tsetlin chain of subalgebras. Let
$Z_{k,l}=Z(\g_{k,l})$ be the center of the universal enveloping
algebra $U(\g_{k,l})$. A commutative subalgebra $\Gamma\in
U(\g_{m,n})$ generated by the centers $Z_{k,l}$ corresponding to
the members of the Gelfand-Tsetlin chain is called a
\emph{Gelfand-Tsetlin subalgebra}. 
We will say that a Gelfand-Tsetlin subalgebra is \emph{special} if it contains  a Gelfand-Tsetlin
subalgebra of $\mathfrak{gl}(n)$ or $\mathfrak{gl}(m)$.

 Consider a Gelfand-Tsetlin chain in $\g_{m,n}$ and a
corresponding Gelfand-Tsetlin subalgebra $\Gamma$. The
Gelfand-Tsetlin chain of $\Gamma$ induces the Gelfand-Tsetlin chain
in the even part $\g_{0}$ with subalgebras $(\g_{k,l})_0$.

The centers $Z((\g_{k,l})_0)$  of the universal enveloping
algebras $U((\g_{k,l})_{0})$ generate a
 Gelfand-Tsetlin subalgebra $\Gamma_0$ of
$U(\g_{0})$. 

Recall the homomorphism $\phi: U(\g)^d \rightarrow U(\g_{0})$. Since $\Gamma\subset U(\g)^d$ 
we can consider the restriction
$\phi(\Gamma)\subset U(\g_0)$. Since $\phi(Z_{k,l}\subset Z((\g_{k,l})_0)$ we get $\phi(\Gamma)\subset\Gamma_0$.

\begin{proposition}\label{prop-HC-surj}
The restriction of $\phi:\Gamma\to\Gamma_0$ is surjective, that is
$\phi(\Gamma)= \Gamma_0$.

\end{proposition}

\begin{proof}
We will prove the statement by induction on the rank $m+n$. The base of induction is trivial. 
Suppose that $\g'=\g_{m-1,n}$ is the member 
of the Gelfand-Tsetlin chain, $\Gamma'\subset \Gamma$ is the
corresponding Gelfand-Tsetlin subalgebra of $U(\g')$ and
$\Gamma_0'\subset \Gamma_0$ is the corresponding Gelfand-Tsetlin
subalgebra of $U(\g'_0)$.
 By induction assumption
$\phi(\Gamma')=\Gamma'_0$. This implies, in particular, that $\phi(\Gamma')$ contains
$Z(\mathfrak{gl}(n))$. We need now the following

\begin{lemma} $Z(\mathfrak{gl}(n))$ and
$\phi(Z_{m,n})$ generate the center  of $U(\g_0)$.
\end{lemma}
\begin{proof} Recall the description of the image under the Harish-Chandra homomorphism 
$HC:=\phi_0\phi$ given in the the end of the previous section. Using formula (\ref{hc}) and the fact that
$HC(Z(\mathfrak{gl}(n)))$ is generated by coefficients of $(1+tx_{m+1})\cdots(1+tx_{m+n})$ we obtain that
$ HC(Z(\mathfrak{gl}(n)))$ and $HC(Z_{m,n})$ generate $HC(Z(\g_0))$. Since
$\phi_0: Z(\g_0)\to  S(\mathfrak h)$
is injective, we obtain the statement of Lemma.
\end{proof}
 Hence,
$\phi(Z_{m,n})$ and $\Gamma'_0$ generate  $Z(\g_0)$ and thus
$\Gamma_0$. We conclude that $\phi$ is surjective.

\end{proof}

Note that the restriction $\phi:Z(\g_{m,n})\to Z(\g_0)$
is injective but not surjective. In contrast the restriction of $\phi$
to $\Gamma$ is surjective but not injective as one can see from the following example.

\begin{example}
Let $\g=\mathfrak{gl}(1|1)$. In this case there exist two Gelfand-Tsetlin
chains $\mathfrak{gl}(1|1)\supset \mathfrak{gl}(1|0)$ and $\mathfrak{gl}(1|1)\supset
\mathfrak{gl}(0|1)$. Both chains define the same Gelfand-Tsetlin subalgebra
$\Gamma$ generated by $E_{00}$, $E_{11}$ and $x=E_{10}E_{01}\in\operatorname{Ker}\phi_0$.
Observe that we have the relation
$$x(x-E_{00}-E_{11})=0.$$

\end{example}

The homomorphism $\phi:\Gamma\to\Gamma_0$ induces the dual
map
$$\phi^*: \hat{\Gamma}_0 \rightarrow \hat{\Gamma},$$ 
which is injective by the above proposition.


\begin{remark}
Since $\Gamma_0$ is a polynomial algebra there exists an injective
homomorphism $\psi:\Gamma_0\rightarrow \Gamma$ such that $\phi\psi=\operatorname{id}$
and $\psi((\Gamma_{k,l})_0)\subset \Gamma_{k,l}$.
\end{remark}

\begin{remark} It is not difficult to see that if $\g=\mathfrak{gl}(n|1)$ and $\Gamma$ is the Gelfand-Tsetlin subalgebra associated with the chain
  $$\mathfrak{gl}(1)\subset \mathfrak{gl}(2)\subset\dots\subset\mathfrak{gl}(n) \subset \mathfrak{gl}(n|1)$$
  then $\Gamma$ contains $\Gamma_0$. In general, $\Gamma_0$ is not a subalgebra of $\Gamma$.
    \end{remark}

\begin{example}  Consider the chain of superalgebras:
  $$\mathfrak{gl}(1)\subset\mathfrak{gl}(1|1)\subset\mathfrak{gl}(1|2).$$
  Let us see that $\Gamma$ does not contain $\Gamma_0$, moreover, these two subalgebras do not commute.
  Indeed, by the previous example $x=E_{10}E_{01}$ is an element of $\Gamma$ and $y=E_{32}E_{23}$ is an element of $\Gamma_0$, since the Casimir element
  of $U(\mathfrak{gl}(2))$ lies in $U(\mathfrak h)+2y$. A simple computation shows that
  $$[x,y]=E_{31}E_{23}E_{12}-E_{21}E_{32}E_{13}\neq 0.$$
  \end{example}

\

\section{Gelfand-Tsetlin modules for $\gl(m|n)$}

\subsection{Gelfand-Tsetlin modules over $\g_{0}$}

Let $\g=\mathfrak{gl}(m|n)=\g{-1}\oplus\g_{0}\oplus \g_1$. 
Let $\Gamma_0$ be a Gelfand-Tsetlin subalgebra of $U(\g_0)\simeq U(\mathfrak{gl}(m)\oplus \mathfrak{gl}(n))$.
We will denote by $\hat{\Gamma}_0=Hom(\Gamma_0, \C)$ the space of characters of $\Gamma_0$.

 The
following is straightforward.

\begin{proposition} There exist a Gelfand-Tsetlin subalgebra $\Gamma(m)$ of $U(\mathfrak{gl}(m))$ and a Gelfand-
Tsetlin subalgebra $\Gamma(n)$ of $U(\mathfrak{gl}(n))$ such that $\Gamma_0\simeq \Gamma(m)\otimes \Gamma(n)$. 
 \end{proposition}
 
 Thus combining Gelfand-Tsetlin subalgebras of $U(\mathfrak{gl}(m))$ and $U(\mathfrak{gl}(n))$ we obtain all Gelfand-Tsetlin subalgebras of 
$U(\g_{0})$.

A finitely generated module $M$ over $\g_0$ is called a {\em
Gelfand-Tsetlin module\/} (with respect to $\Gamma_0$) if \begin{equation*}
M=\underset{\chi\in \hat{\Gamma}_0}{\bigoplus}M({\chi}), \end{equation*} where \begin{equation*} M({\chi}) \ = \ \{ x\in M\ | \  (z-\chi(z))^k x
=0\quad \text{for some}\quad k\geqslant 0  \quad \text{and all} \quad  z\in \Gamma_0   \}. \end{equation*}

Denote by $\mathcal M(\Gamma_0)$ the category of Gelfand-Tsetlin modules  over $\g_{0}$ with respect to the subalgebra $\Gamma_0$.
Clearly, all finite dimensional $\g_{0}$-modules belong to $\mathcal M(\Gamma_0)$. 


\begin{theorem}\label{g0-GT-finite-length} Let $V$ be a  finite dimensional $\g_{0}$-module. For $M\in \mathcal M(\Gamma_0)$ denote 
$F_V(M)=M\otimes V$. Then the correspondance
 $F_V: M \rightarrow F_V(M)$ defines a  endofunctor of the category $\mathcal M(\Gamma_0)$.


\end{theorem}

\begin{proof} First, note that $F_V(M)$ is finitely generated for any $M\in \mathcal M(\Gamma_0)$.
  Indeed, it suffices to check this for a cyclic module $M$. But in this case $M$ is a quotient of a free module $U(\g_0)$. Then
  $M\otimes V$ is a quotient of a free module of rank $\dim V$ and hence finitely generated.
  
Let $Z_k$ be the center of $U(\gl(k))$ and suppose that $Z_k\subset\Gamma_0$.  To prove the statement it is sufficient to show that 
$Z_k$ acts locally finitely on $F_V(M)$. Consider both $M$ and $V$ as $\gl(k)$-modules. 
Tensor product with $V$ defines a projective functor on the category of $\gl(k)$-modules which restricts to the functor 
on the subcategory 
with locally finite action of the center $Z_k$ by [\cite{BG}, Corollary 2.6] or [\cite{K1}, Theorem 5.1]. Since $Z_k$ is locally finite on $M$, it will also 
act   locally finitely on $F_V(M)$. 
In particular, $Z(\g_{\bar{0}})$ acts locally finitely on $V\otimes W$. The statement follows.


\end{proof}

Gelfand-Tsetlin theory for $\gl(n)$ was developed in \cite{O}. The results can be easily extended to Gelfand-Tsetlin modules over $\g_{\bar{0}}$.

\begin{proposition}\label{prop-fin-mult} Let  $V$ be some finite-dimensional $\g$-module.
\item[(i)] Let $\chi\in\hat{\Gamma}_0$. There exists a finite set $S(\chi)\subset\hat{\Gamma}_0$ such that for any Gelfand-Tsetlin module
  $M$ we have
  $$M(\chi)\otimes V\subset\bigoplus_{\theta\in S(\chi)}[M\otimes V](\theta).$$

\item[(ii)] For every $\chi\in\hat{\Gamma}_0$ the set $S^{-1}(\chi):=\{\theta\mid \chi\in S(\theta)\}$ is finite.
  
\item[(iii)] If $\dim M(\chi)<\infty$ for all $\chi\in\hat{\Gamma}_0$ then   $\dim [M\otimes V](\theta)<\infty$  for all $\theta\in\hat{\Gamma}_0$
  \end{proposition}

\begin{proof} Let us prove (i). The proof goes by induction on the rank of $\g_0$. We assume that the statement is true for the previous term
  $(\g'_0,\Gamma'_0)$ in the Gelfand-Tsetlin chain. Let $\chi'$ be the restriction of $\chi$ to $\Gamma'_0$ and $\xi$ be the central character obtained by
  restriction of $\chi$ to $Z(\g_0)$. Then $M(\chi)=M(\chi')\cap M(\xi)$. We have
  $$M(\chi)\otimes V=M(\chi')\otimes V\cap M(\xi)\otimes V.$$
  By the Kostant theorem $$M(\xi)\otimes V=\oplus_{\zeta\in S(\xi)}[M\otimes V](\zeta)$$
  for some finite set $S(\xi)$. Hence
  $$M(\chi)\otimes V\subset \bigoplus_{\zeta\in S(\xi),\theta'\in S(\chi')}[M\otimes V](\zeta)\cap [M\otimes V](\theta')=\bigoplus_{\theta\in S(\chi)}[M\otimes V](\theta),$$
  where $S(\chi)$ consists of all characters $\theta$ such that its restriction to $\Gamma'_0$ lies in $S(\theta')$ and the restriction to $Z(\g_0)$ lies in $S(\xi)$.

  Now (ii) is a consequence of the proof of (i) since $S^{-1}(\xi)$ is finite for any central character $\xi$.

  Finally, (iii) follows immediately for (ii) and (i).
  \end{proof}

\subsection{Gelfand-Tsetlin $\g$-modules }
Let $\Gamma$ be a Gelfand-Tsetlin subalgebra of $\g=\mathfrak{gl}(m|n)$, $\g=\g_{-1}\oplus \g_{0}\oplus \g_{1}$.  We will denote by $\hat{\Gamma}=Hom(\Gamma, \C)$ the space of characters of $\Gamma_0$.

A finitely generated $\g$-module $M$  is called a {\em
Gelfand-Tsetlin module\/} (with respect to $\Gamma$) 
if
\begin{equation*}
M=\underset{\chi\in \hat{\Gamma}}{\bigoplus}M({\chi}), \end{equation*} where \begin{equation*} M({\chi}) \ = \ \{ x\in M\ | \  (z-\chi(z))^k x
=0\quad \text{for some}\quad k\geqslant 0  \quad \text{and all} \quad  z\in \Gamma   \}. \end{equation*}

Denote by $\mathcal M(\Gamma)$ the category of Gelfand-Tsetlin modules  over $\g$ with respect to the subalgebra $\Gamma$.
In particular, all finite dimensional $\g$-modules belong to $\mathcal M(\Gamma)$. Let $\mathcal F(\Gamma)$ be the category of $\g$-modules with locally
finite action of $\Gamma$ and $\mathcal F_{\g}(\Gamma_0)$ be the category of $\g$-modules with locally finite action of $\Gamma_0$. Obviously, $\mathcal M(\Gamma)$ 
is a full subcategory of $\mathcal F(\Gamma)$. Moreover,
every object of $\mathcal F(\Gamma)$ is a direct limit of objects in $\mathcal M(\Gamma)$.

\begin{lemma}\label{lem-center0-center} If $Z(\g_0)$ acts locally finitely on a $\g$-module $M$ then $Z(\g)$ acts locally finitely on $M$. 

\end{lemma}

\begin{proof} 
  Let $N:=M^{\g_1}$. Obviously, $N$ is a $\g_0$-submodule of $M$. Since $zv=\phi(z)v$ for any $z\in Z(\g)$ and any $v\in N$, and $\phi(Z(\g))\subset Z(\g_0)$,
  we obtain that 
 $Z(\g)$ acts  locally finitely on  $N$ and hence on $U(\g)N$. 

  Consider a filtration
  $$0\subset F^1(M)\subset F^2(M)\subset \dots\subset F^k(M)\subset \dots,$$
  defined inductively by
  $$F^i(M)/F^{i-1}M=U(\g)(M/F^{i-1}M)^{\g_1}.$$
  By repeating the above argument we get that $Z(\g)$ acts locally finitely on $F^i(M)$ for all $i$.

  On the other hand, $\dim U(\g_1)v\leq 2^{mn}$ for any $v\in M$, hence $M=\cup_{i=1}^{2^{mn}} F^i(M)$ and the statement follows.

\end{proof}

\begin{theorem}\label{thm-gamma0-gamma} Let $\Gamma$ be a Gelfand-Tsetlin subalgebra of $U(\g)$ such that
$\Gamma_0=\phi(\Gamma)$. Then 
  $\mathcal F(\Gamma)=\mathcal F_{\g}(\Gamma_0)$ and $\mathcal M(\Gamma)=\mathcal M_{\g}(\Gamma_0)$ where $\mathcal F_{\g}(\Gamma_0)$ (resp.,$\mathcal M_{\g}(\Gamma_0)$)
  is the category of $\g$-modules (resp., finitely generated $\g$-modules) locally finite over $\Gamma_0$.
\end{theorem}

\begin{proof} Suppose $M\in\mathcal F(\Gamma)$. Consider the filtration $F^i(M)$ defined in the proof of Lemma \ref{lem-center0-center}.
  To prove that $M\in\mathcal F_{\g}(\Gamma_0)$ it suffices to show that $\Gamma_0$ acts locally finitely on $F_i:=F^i(M)/F^{i-1}M$.
  Indeed, consider $F_i$ as a $\g_0$-module. Then $F_i$ is a quotient of $\Lambda(\g_{-1})\otimes (M/F^{i-1}M)^{\g_1}$. For any $\gamma\in\Gamma$ and
  $v\in(M/F^{i-1}M)^{\g_1}$ we get $\gamma v=\phi(\gamma)v$. Since $\phi:\Gamma\to\Gamma_0$ is surjective by Proposition \ref{prop-HC-surj}, we obtain that
  $\Gamma_0$ acts locally finitely on $(M/F^{i-1}M)^{\g_1}$. Furthermore, $F^i(M)/F^{i-1}M$ considered as a $\g_0$-module is a quotient of
  $\Lambda(\g_{-1})\otimes (M/F^{i-1}M)^{\g_1}$. Hence by Theorem \ref{g0-GT-finite-length} we obtain that $\Gamma_0$ acts locally finitely on $F^i(M)/F^{i-1}M$.

  Suppose now that $M\in\mathcal F_{\g}(\Gamma_0)$. By Lemma \ref{lem-center0-center} we get that $Z(\g_{k,l})$ acts locally finitely on $M$.
  Therefore $\Gamma$ acts locally finitely on $M$.

  The second assertion is a consequence of the obvious fact that $U(\g)$ is a finitely generated as $U(\g_0)$-module and hence a $\g$-module
  $M$ is finitely generated over $U(\g_0)$ if and only if it is finitely generated over
  $U(\g)$. 
\end{proof}

As a consequence we obtain the following super analog of Theorem \ref{g0-GT-finite-length}.

\begin{corollary}\label{g-GT-finite-length}
Let $V$ be a  finite dimensional $\g$-module. For $M\in \mathcal F(\Gamma)$ set
$\hat{F}_V(M)=M\otimes V$. Then the correspondence
 $\hat{F}_V: M \rightarrow \hat{F}_V(M)$ defines a functor on the category $\mathcal F(\Gamma)$.
The restriction of $\hat{F}_V$ on $\mathcal M(\Gamma)$ defines a functor on the category $\mathcal M(\Gamma)$.
\end{corollary}

\begin{proof}
Indeed, we have $M\in \mathcal F(\Gamma)=\mathcal F(\Gamma_0)$ by Theorem \ref{thm-gamma0-gamma}. Hence $\hat{F}_V(M)\in \mathcal F(\Gamma_0)$ by Theorem \ref{g0-GT-finite-length}. Applying Theorem \ref{thm-gamma0-gamma} again we conclude that $\hat{F}_V(M)\in \mathcal F(\Gamma)$.
\end{proof}

We also have the following surprizing result

\begin{corollary}\label{cor-indep-gamma}
If $\Gamma$ and $\Gamma'$ are two Gelfand-Tsetlin subalgebras of $\g$ such that  $\phi(\Gamma)=\phi(\Gamma')$, 
 then
$\mathcal F(\Gamma)=\mathcal F(\Gamma')$.
\end{corollary}

\begin{proof}
Indeed, by Theorem \ref{thm-gamma0-gamma} we have $$\mathcal F(\Gamma)=\mathcal F(\phi(\Gamma))=\mathcal F(\phi(\Gamma'))= \mathcal F(\Gamma').$$
\end{proof}

Let $M$ be a $\g$-module. We will say that $v\in M$ is a
Gelfand-Tsetlin vector with respect to $\Gamma$ of weight $\chi$
if $v\in M(\chi)$. We have the following key property of Gelfand-Tsetlin modules.

\begin{proposition}\label{prop-GTs}
Let $M$ be a $\g$-module, $\Gamma$ a Gelfand-Tsetlin
subalgebra of $U(\g)$, $\chi\in \hat{\Gamma}$, $M(\chi)\neq 0$. Assume that $M$ is generated by $M(\chi)$.
 Then $M\in \mathcal F(\Gamma)$. Furthermore, if $\dim M(\chi)<\infty$ then $M$ is a Gelfand-Tsetlin module with respect to
$\Gamma$.

\end{proposition}

\begin{proof} We prove the statement by induction in $m+n$.
  Let $\g'=\g_{m,n-1}$ or $\g_{m-1,n}$ be the previous term in the defining chain.
  By induction assumption $M':=U(\g')M(\chi)\in\mathcal F(\Gamma')$. Furthermore $M$ is a quotient of the induced module
  $U(\g)\otimes_{U(\g')}M'$. The latter is isomorphic to $S(\g/\g')\otimes M'$ as a $\g'$-module. By Corollary \ref{g-GT-finite-length}
  $U(\g)\otimes_{U(\g')}M'\in \mathcal F(\Gamma')$ and hence $M\in \mathcal F(\Gamma')$. On the other hand,
  $Z_{m,n}$ is locally finite on $M$. Hence $M\in  \mathcal F(\Gamma)$.

\end{proof}

\begin{corollary}\label{cor-lifting}
 Let  $M$ be a $\g$-module,   $\phi(\Gamma)=\Gamma_0$,  $\chi_0\in \hat{\Gamma}_0$.
If  $M(\chi_0)\neq 0$  then $M(\chi)\neq 0$ for some $\chi\in \hat{\Gamma}$. If in addition $V$ is generated by $V(\chi_0)$ then $V\in \mathcal F(\Gamma)$.
Furthermore, if $M$ is an irreducible $\g$-module and $M(\chi_0)\neq 0$ then $M\in \mathcal M(\Gamma)$.
\end{corollary}

\begin{proof} Consider the $\g$-submodule
  $N$ of $M$ generated by $M(\chi_0)$. Let $N'=U(\g_0)M(\chi)$.  Then
  $N$ is a quotient of the induced module $U(\g)\otimes_ {U(\g_0)}N'$.
  Hence 
  $N',N\in \mathcal F(\Gamma_0)$ by Theorem \ref{g0-GT-finite-length}. Applying Theorem \ref{thm-gamma0-gamma} we
  conclude that $N\in\mathcal F(\Gamma)$. 
    Therefore, $N(\chi)\neq 0$ for  some $\chi\in \hat{\Gamma}$ implying the first assertion. Two other assertions are straightforward.
\end{proof}

\section{Classification of Gelfand-Tsetlin modules }

\subsection{Restriction and induction functors} 
Denote by $\bf {K_0}$ the restriction of the functor $\bf K$ introduced above to the category of Gelfand-Tsetlin $\g_0$-modules $\mathcal M(\Gamma_0)$. 

Let $N\in \mathcal M(\Gamma_0)$ and $M=K(N)$. Then $M$ is generated by the $\g_0$-submodule $M_0\subset M^{\g_1}$ isomorphic to $N$. For $v\in M_0$ we have 
$zv=\phi(z)v$ for all $z\in \Gamma 
$, where $\Gamma$  is any Gelfand-Tsetlin subalgebra of $\g$ such that
$\Gamma_0=\phi(\Gamma)$.
Hence 
$M_0$  is a Gelfand-Tsetlin $\g_0$-module with respect to
$\Gamma_0$ and
$$M_0=\oplus_{\chi\in \hat{\Gamma}}M_0(\chi)=\oplus_{\chi_0\in \hat{\Gamma}_0}M_0(\chi_0).$$

Applying Corollary \ref{cor-lifting} we conclude that $M\in \mathcal F(\Gamma)$.  Moreover, since $M$ is finitely generated (as $N$ is finitely generated) 
we have that $M=K(N)$  is a Gelfand-Tsetlin $\g$-module with respect to $\Gamma$. Hence   $\bf {K_0}$ defines a functor:
 $$\bf {K_0}:\mathcal M(\Gamma_0) \rightarrow \mathcal M(\Gamma).$$

Moreover we have:

\begin{corollary}\label{cor-irr-K(N)}
 Let  $M$ be an irreducible  $\g$-module and $M(\chi_0)\neq 0$ for some $\chi_0\in \hat{\Gamma}_0$.
Then
$M$ is a unique irreducible quotient of $K(N)$ 
for some irreducible Gelfand-Tsetlin (with respect to $\Gamma_0$) $\g_0$-module   $N$.
Furthermore, $M$ is a Gelfand-Tsetlin $\g$-module with respect to any $\Gamma$ such that
$\Gamma_0=\phi(\Gamma)$.
\end{corollary}

We also have

\begin{corollary}\label{prop-I(M)} 
\begin{itemize}
\item[(i)]
Given any $\chi_0\in \hat{\Gamma}_0$
 there exists an irreducible $M\in \mathcal M(\Gamma)$ with $M(\chi_0)\neq 0$.
 \item[(ii)]  If $M_0$ is a Gelfand-Tsetlin module for $\g_{{0}}$ with respect to a
Gelfand-Tsetlin subalgebra $\Gamma_0$, then $\Ind_{\g_{{0}}}^{\g}M_0$ is a
Gelfand-Tsetlin module for $\g$ with respect to any
Gelfand-Tsetlin subalgebra $\Gamma$ such that
$\Gamma_0=\phi(\Gamma)$.
\end{itemize}
 \end{corollary}

\begin{proof} Let $N$ be an irreducible $\g_0$-module such that $N(\chi_0)\neq 0$ and $M$ be the unique simple quotient of $K(N)$.
  Then $M(\chi_0)\neq 0$ and the first statement follows.  

Since $M=\Ind_{\g_{{0}}}^{\g}M_0$ is generated by $M_0$ then it is a Gelfand-Tsetlin $\g_0$-module with respect to $\Gamma_0$ by Theorem \ref{g0-GT-finite-length}. Then $M\in \mathcal F(\Gamma)$ by Theorem \ref{thm-gamma0-gamma}. Moreover, $M$ is finitely generated as $M_0$ is finitely generated  as a $\g_0$-module. Hence, $M$
is a
Gelfand-Tsetlin module for $\g$ with respect to any
Gelfand-Tsetlin subalgebra $\Gamma$ such that
$\Gamma_0=\phi(\Gamma)$.
\end{proof}

\

Denote by $\bf R_0$ the restriction of the functor $\bf R$ to the Gelfand-Tsetlin category $ \mathcal M(\Gamma)$. 
For a $\g$-module  
$M$,  $R(M)=M^{\g_1}$ is a $\g_{{0}}$-module.
Since
$\phi:\Gamma\rightarrow \Gamma_0$ is an epimorphism then
$R(M)$  is a Gelfand-Tsetlin $\g_{{0}}$-module with respect to
$\Gamma_0$, hence
 $${\bf R_0}: \mathcal M(\Gamma)\rightarrow \mathcal M(\Gamma_0).$$  By Proposition \ref{prop-adjoint}, the pair $(\bf K_0, \bf R_0)$ is adjoint.



\

\subsection{Reduction to the even part}

The results obtained in the previous section allow to reduce  the classification problem of irreducible Gelfand-Tsetlin $\g$-modules to the classification  of irreducible Gelfand-Tsetlin $\g_0$-modules. 

\begin{theorem}\label{thm-GTs-adj}

\item[(i)] For any irreducible Gelfand-Tsetlin (with respect to $\Gamma$) $\g$-module $M$  there exists a unique
 irreducible Gelfand-Tsetlin (with respect to $\Gamma_0$)  $\g_{{0}}$-module $N$ such that $M$ is a unique irreducible quotient
of $K(N)$.
\item [(ii)] Let $M$ be  an irreducible Gelfand-Tsetlin module with respect to
a Gelfand-Tsetlin subalgebra $\Gamma'$ such that
$\phi(\Gamma')=\Gamma_0$. Then $M$ is a Gelfand-Tsetlin module
with respect to any Gelfand-Tsetlin subalgebra $\Gamma$ such that
$\phi(\Gamma)=\Gamma_0$.

\end{theorem}

\begin{proof}
The first statement follows from Corollary \ref{cor-irr-K(N)} and Theorem \ref{thm-gamma0-gamma}. The second statement follows from Corollary \ref{cor-indep-gamma}.

\end{proof}

 It is shown in \cite{O} (see also \cite[Theorem 4.12(c)]{FO}) that  Gelfand-Tsetlin multiplicities of any  irreducible Gelfand-Tsetlin $\g_0$-module are finite and uniformly bounded. Then we have

\begin{corollary}\label{fin-mult} If $M$ is an irreducibe Gelfand-Tsetlin module. Then Gelfand-Tsetlin multiplicities of $M$ are finite and uniformly bounded.
\end{corollary}
\begin{proof} By Theorem \ref{thm-GTs-adj} (i) it suffices to check this property for $K(N)$ for some irreducible Gelfand-Tsetlin module
  $N$. The Gelfand-Tsetlin multiplicities of $N$ are finite and uniformly bounded.
   Since $K(N)\simeq \Lambda(\g_{-1})\otimes N$ as
  a $\g_0$-module, the statement follows
  from Proposition \ref{prop-fin-mult} (iii).
  \end{proof}

Finally we have the following finiteness result.

\begin{theorem}\label{theorem-finiteness} Let $\Gamma$ be a  Gelfand-Tsetlin subalgebra of $U(\g)$. 
\begin{itemize}
  \item[(i)] Given
$\chi_0\in\hat{\Gamma}_0$, there exist  finitely
many non-isomorphic irreducible Gelfand-Tsetlin (with respect to
$\Gamma$) $\g$-modules
 $M$ such that  $M^{\g_{1}}\cap M(\chi_0)\neq 0$. 
\item[(ii)] For any $\chi_0\in \hat{\Gamma}_0$ there exist  finitely many
non-isomorphic irreducible Gelfand-Tsetlin (with respect to
$\Gamma$) $\g$-modules $M$ with $M(\chi_0)\neq 0$.

\end{itemize}
\end{theorem}

\begin{proof} The first statement is an immediate consequence of Theorem  \ref{thm-GTs-adj} (i).

  By the finiteness theorem for $\g_0$ (\cite{O}), for any $\theta\in \hat{\Gamma}_0$ there are finitely many up to isomorphism irreducible Gelfand-Tsetlin modules $N$
  such that $N(\theta_0)\neq 0$. We claim that there are finitely many $N$ such that $K(N)(\chi_0)\neq 0$. For this we recall that $K(N)$ is isomorphic to
  $N\otimes V$ for $V=\Lambda(\g_{-1})$. By  Proposition \ref{prop-fin-mult} (i), $K(N)(\chi_0)\neq 0$ implies that $N(\theta_0)\neq 0$ for some
  $\theta_0\in S^{-1}(\chi_0)$. By Proposition \ref{prop-fin-mult} (ii), the set $S^{-1}(\chi)$ is finite, hence the statement follows from the above
  mentioned result for
  $\g_0$-modules.
\end{proof}

\

\subsection{ Equivalence of subcategories  }

Recall that  the category of representations
of a basic classical Lie superalgebra  of type I with a fixed typical central character is
equivalent to the category of representations of its even part with a suitable central
character \cite{PS},\cite{G}. Due to Proposition \ref{prop-I(M)}  this equivalence restricts to the Gelfand-Tsetlin categories $\mathcal M(\Gamma)$ and $\mathcal M(\Gamma_0)$.  Namely we have the following.

Let $\xi\in \hat{Z}(\g)$ and $\mathcal M^{\xi}(\Gamma)$ denote the full subcategory of $\mathcal M(\Gamma)$ consisting of modules admitting
the central character $\xi$. Let $\xi_0:=\phi^*(\xi)$ and  $\mathcal M^{\xi_0}(\Gamma_0)$ the subcategory of Gelfand-Tsetlin $\g_0$-modules which admit the central
character $\xi_0$.

Recall that $\xi$ is typical if  $\xi$ is the central character of a Verma module with
typical highest weight.

It is shown in \cite{PS} that if $\xi$ is typical then the functors $\bf R$ and $\bf K$ define equivalence between the categories of $\g$-modules admitting $\xi$ and
the category of $\g_0$-modules admitting $\xi_0$. Since $\bf K$ and $\bf R$ restrict to the functors
$$\bf K_0:\mathcal M^{\xi_0}(\Gamma_0)\to \mathcal M^{\xi}(\Gamma),\quad \bf R_0:\mathcal M^{\xi}(\Gamma)\to \mathcal M^{\xi_0}(\Gamma_0),$$
we obtain the following theorem.

\begin{theorem}\label{thm-equiv}
If  $\xi\in \hat{Z}(\g)$ is typical then $\bf K_0$ and $\bf R_0$ define the equivalence between the categories $\mathcal M^{\xi}(\Gamma)$ and  $\mathcal M^{\xi_0}(\Gamma_0)$.
\end{theorem}

\

\section{Gelfand-Tsetlin bases}
Given a tuple $\la$ of complex numbers of the form $\la=(\la_{1},\la_{2},\ \ldots,\la_{m+n})$, the irreducible highest weight module $L(\lambda)$ of the Lie superalgebra $\gl(m|n)$ with the highest weight $\lambda$ is generated by a nonzero vector $v$ satisfying the conditions:

\begin{equation}
\begin{split}
&h_i v=\lambda_i v, 1\leq i\leq m+n, \\
&e_i v=0,  1\leq i\leq m+n-1.
\end{split} 
\end{equation}

Recall that the  weight $\lambda$ is dominant if 
\begin{equation}
\la_i-\la_{i+1}\in \mathrm{Z}_{\geq 0}, \text{ for }i=1, \ldots,m+n-1,  \quad i\neq m.
\end{equation}
The dominant weight $\la$ (and the module $L(\la)$) is typical if  
$(\lambda+\rho,\varepsilon_i-\delta_j)\neq 0$ for  $1\leq i\leq m,1\leq j\leq n$.

Set
\begin{equation}
l_{i}=\la_{i}-i+1, (1\leq i\leq m); \quad l_{j}=-\la_{j}+j-2m, (m+1\leq j\leq m+n).
\end{equation}
Since $(\lambda+\rho,\varepsilon_i-\delta_j)=l_{i}-l_{j}$,  
the condition for typicality is equivalent to $l_{i}\neq l_{j}$ for  $1\leq i\leq m <j\leq m+n. $

For $a, b \in \C$ such that $b-a$ is nonnegative integer, we set 
$$[a;b]=\{a+s \ | 0\leq s\leq b-a\}.$$

The dominant weight $\la$ (and the module $L(\la)$) is called \emph{essentially typical} if  
\begin{equation}
\{l_1,l_2,\ldots,l_m\}\cap [l_{m+1}; l_{m+n}]=\emptyset.
\end{equation}

Clearly,  every essentially typical weight is typical.

%


Our main combinatorial device is an array of complex numbers 
 $\Lambda=(\lambda_{ij})$ presented in the following form:
 
 \begin{equation}\label{GT pattern}
\begin{array}{cccccccc}
   \la_{m+n,1}   &  \cdots      & \la_{m+n,m} & \la_{m+n,m+1} & \cdots & \la_{m+n-1,m+n} & \la_{m+n,m+n} \\
  \la_{m+n-1,1} &  \cdots     & \la_{m+n-1,m} & \la_{m+n-1,m+1} & \cdots & \la_{m+n-1,m+n-1} & \\
  \vdots    &  \vdots & \vdots & \vdots   & \reflectbox{$\ddots$}    \\
  \la_{m+1,1} &  \cdots  & \la_{m+1,m} & \la_{m+1,m+1}    \\
  \la_{m1} &  \cdots  & \la_{mm}                           \\
  \la_{m-1,1} &  \cdots                             \\
  \vdots  &  \reflectbox{$\ddots$}                              \\
  \la_{11}\\
\end{array}
\end{equation}
 We will call  $\Lambda$ a  \emph{Gelfand-Tsetlin tableau}.  Given $\Lambda=(\lambda_{ij})$
we set

 \begin{equation}
l_{ki}=\la_{ki}-i+1, (1\leq i\leq m); \quad l_{kj}=-\la_{kj}+j-2m, (m+1\leq j\leq k).
\end{equation}

For  an essentially typical highest weight $\la$  a basis of the irreducible module $L(\la)$
  parameterized by Gelfand-Tsetlin tableaux was constructed in   \cite{Palev1989a}. We will use the following modified version of these formulas.

\begin{theorem}
Let $\la$ be an essentially typical highest weight. Then $L(\la)$
admits a basis $\xi_{\La}$ parameterized by all Gelfand-Tsetlin  tableaux $\La$ which satisfy the
following conditions:
\begin{enumerate}
  \item $\lambda_{m+n,i}=\lambda_i, \quad 1\leq i\leq m+n$;
  \item  $\la_{k,i}-\la_{k-1,i}\equiv\theta_{k-1,i}\in\{0,1\}, 1\leq i\leq m;m+1\leq k\leq m+n$;
  \item $\la_{ki}-\la_{k,i+1}\in \mathrm{Z}_{\geq 0}, 1\leq i\leq m-1;m+1\leq k\leq m+n-1$;
  \item  $\la_{k+1,i}-\la_{ki}\in \mathrm{Z}_{\geq 0}$ {\it and} $\la_{k,i}-\la_{k+1,i+1}\in \mathrm{Z}_{\geq 0},$
$1\leq i\leq k\leq m-1$ \text{ or } $m+1\leq i\leq k\leq m+n-1$.
\end{enumerate}
The action of the generators of $\gl_{m,n}$
is given by the formulas
\begin{equation}\label{action h}
h_{k}\xi_{\La}=\left(\sum_{i=1}^{k}\la_{kj}-\sum_{j=1}^{k-1}\la_{k-1,j}\right)\xi_{\La},\quad 1\leq k\leq m+n;
\\
\end{equation}

\begin{equation}
\noindent e_{k}\xi_{\La}=-\sum_{i=1}^{k}\frac{\Pi_{j=1}^{k+1}(l_{k+1,j}-l_{ki}) }
  {\Pi_{j\neq i,j=1}^{k} (l_{kj}-l_{ki}) }\xi_{\La+\delta_{ki}},
\quad1\leq k\leq m-1;
\end{equation}

\begin{equation}
f_{k}\xi_{\La}=\sum_{i=1}^{k}\frac{\Pi_{j=1}^{k-1}(l_{k-1,j}-l_{ki})}
{\Pi_{j\neq i,j=1}^{k}(l_{kj}-l_{ki})}\xi_{\La-\delta_{ki}},
\quad  1\leq k\leq m-1;
\end{equation}

\begin{equation}
\begin{split}
e_{m}\xi_{\La}=\sum_{i=1}^{m}\theta_{mi}(-1)^{i-1}(-1)^{\theta_{m1}+\ldots+\theta_{m,i-1}}
\\
\times  \frac{\Pi_{1\leq j< i} (l_{mj}-l_{mi}-1)}
  {\Pi_{i<j\leq m} (l_{mj}-l_{mi})
    \Pi_{j\neq i,j=1}^{m}(l_{m+1,j}-l_{mi}-1)}
    \xi_{\La+\delta_{mi}},
  \\
\end{split}
\end{equation}

\begin{equation}
\begin{split}
f_{m} \xi_{\La}=\sum_{i=1}^{m}(1-\theta_{mi})(-1)^{i-1}(-1)^{\theta_{m1}+\ldots+\theta_{m,i-1}}
\\
\times  \frac{(l_{m,i}-l_{m+1,m+1})\Pi_{ i<j\leq m} (l_{mj}-l_{mi}+1)\Pi_{j=1}^{m-1}(l_{m-1,j}-l_{mi})}
  {\Pi_{1\leq j< i} (l_{mj}-l_{mi})} \xi_{\La-\delta_{mi}},
\end{split}
\end{equation}

\begin{equation}
\begin{split}
e_k \xi_{\La}=\sum_{i=1}^{m}\theta_{ki}(-1)^{\theta_{k1}+\ldots+\theta_{k,i-1}
+\theta_{k-1,i+1}+\ldots+\theta_{k-1,m}}(1-\theta_{k-1,i})
\\
\times
\prod_{j\neq i,j =1}^{m}\left(\frac{l_{kj}-l_{ki}-1}{l_{k+1,j}-l_{ki}-1}\right)
\xi_{\La+\delta_{ki}}
\\
\\
-\sum_{i=m+1}^{k} 
\Pi_{j=1}^{m}\left(\frac{(l_{kj}-l_{ki})(l_{kj}-l_{ki}+1)}{(l_{k+1,j}-l_{ki})(l_{k-1,j}-l_{ki}+1)} \right)
\\
\times
  \frac{\Pi_{j=m+1}^{k+1}(l_{k+1,j}-l_{ki})}
  {\Pi_{j\neq i,j=m+1}^{k} (l_{kj}-l_{ki})}\xi_{\La+\delta_{ki}} ,
\quad  m+1\leq k\leq m+n-1;
\\
\end{split}
\end{equation}

\begin{equation}\label{action f}
\begin{split}
f_{k}\xi_{\La}=
\sum_{i=1}^{m}\theta_{k-1,i}(-1)^{\theta_{k1}+\ldots+\theta_{k,i-1}+\theta_{k-1,i+1}+\ldots+\theta_{k-1,m}}(1-\theta_{ki})
\\
\times
\prod_{j\neq i=1}^{m}\left(\frac{l_{kj}-l_{ki}+1}{l_{k-1,j}-l_{ki}+1}\right)
\frac{\Pi_{j=m+1}^{k+1}(l_{k+1,j}-l_{ki})\Pi_{j=m+1}^{k-1}(l_{k-1,j}-l_{ki}+1)}
  {\Pi_{j=m+1}^{k} (l_{kj}-l_{ki})(l_{kj}-l_{ki}+1)}\xi_{\La-\delta_{ki}}
  \\
  + \sum_{i=m+1}^{k}
\frac{\prod_{j=m+1}^{k-1}(l_{k-1,j}-l_{ki})}{\prod_{j\neq i,j =m+1}^{k}(l_{k,j}-l_{ki})}
\xi_{\La-\delta_{ki}} \quad
 m+1\leq k\leq m+n-1.
\end{split}
\end{equation}

The arrays $\La\pm \delta_{ki}$ are obtained from $\La$ by replacing $\la_{ki}$ by $\la_{ki}\pm1$.  We assume  that
$\xi_{\La}=0$ if the array $\La$ does not satisfy the conditions of the theorem.
\end{theorem}

 The formulas in the theorem above can be obtained from  \cite{Palev1989a} by multiplying the basis elements by $\sqrt{|a(\La)|}$, where $a(\La)=\prod_{k=1}^{m+n}a_k(\La)$.

\begin{equation}
\begin{split}
a_k(\Lambda)=&\prod_{1\leq i\leq j<k}\frac{(l_{ki}-l_{k-1,j})!}{(l_{k-1,i}-l_{k-1,j})!}\prod_{1\leq i<j\leq k}\frac{(l_{ki}-l_{kj}-1)!}{(l_{k-1,i}-l_{kj}-1)!},
\quad 1\leq k\leq m,
\\
a_k(\Lambda)=&
\prod_{1\leq i\leq j\leq m,m+1\leq i\leq j<k}\frac{(l_{k-1,j}-l_{ki})!}{(l_{k-1,j}-l_{k-1,i})!}\\
&\times 
\prod_{1\leq i<j\leq m,m+1\leq i<j\leq k}\frac{(l_{kj}-l_{ki}-1)!}{(l_{kj}-l_{k-1,i}-1)!}
,
\quad m+1\leq k\leq m+n.
\end{split}
\end{equation}

Here the function $a!$ is defined by

\begin{equation}
a!=
\left\{ \begin{aligned}
&1\cdot2 \cdots a, \text{ for } a\in \mathbb Z_{>0} ,\\
&1,   \text{ for } a=0 \text{ or }\pm 1,\\
&\frac{1}{(a+1)(a+2)\cdots(-1)}, \text{ for } a\in \mathbb Z_{<0}.
\end{aligned} \right.
\end{equation}

In the following sections we will identify $ \xi_{\La}$ with the tableau $[l]$
for convenience.

\section{Quasi typical modules}
\subsection{Relation modules for $\gl(n)$} 
Let us recall the construction of  relation modules for $\gl(n)$ \cite{FRZ}.
Set $\mathfrak{V}:=\{(i,j)\ |\ 1\leq j\leq i\leq n\}$ and consider the following subsets of $\mathfrak{V}\times\mathfrak{V}$:
\begin{align}
\mathcal{R}^+ &:=\{((i,j);(i-1,t))\ |\ 2\leq j\leq i\leq n,\ 1\leq t\leq i-1\}\\
\mathcal{R}^- &:=\{((i,j);(i+1,s))\ |\ 1\leq j\leq i\leq n-1,\ 1\leq s\leq i+1\}\\
\mathcal{R}^{0}&:=\{((n,i);(n,j))\ |\ 1\leq i\neq j\leq n\}
\end{align}

Set $\mathcal{R}:=\mathcal{R}^{-}\cup\mathcal{R}^{0}\cup\mathcal{R}^{+}$. Any subset $\mathcal{C}\subseteq \mathcal{R}$ will be called a {\it set of relations}. Let $\mathcal{C}^0=\mathcal{C}\cap \mathcal{R}^{0}$ and $\mathcal{C}^{\pm}=\mathcal{C}\cap \mathcal{R}^{\pm}$.

To any $\mathcal{C}\subseteq \mathcal{R}$ we can associate a directed graph $G(\mathcal{C})$ with the set of vertices  $\mathfrak{V}$ as follows: there is an arrow from $(i,j)$ to $(r,s)$ if and only if $((i,j);(r,s))\in\mathcal{C}$.  Given $(i,j),\ (r,s)\in\mathfrak{V}$ we will write $(i,j)\succeq_{\mathcal{C}}  (r,s)$ if there exists a path in $G(\mathcal{C})$ starting in $(i,j)$ and ending in $(r,s)$.

\begin{definition}
Let $\mathcal{C}$ be any  set of relation and $[l]=(l_{ij})$, $1\leq i\leq n$, $1\leq j\leq i$, a Gelfand-Tsetlin tableau for $\gl(n)$.

\begin{itemize}
\item[(i)] We will say that $[l]$ satisfies $\mathcal{C}$ if:
\begin{itemize}
\item[$\bullet$] $l_{ij}-l_{rs}\in \mathbb{Z}_{\geq 0}$ for any $((i,j); (r,s))\in \mathcal{C}^+\cup\mathcal{C}^0$.
\item[$\bullet$] $l_{ij}-l_{rs}\in \mathbb{Z}_{> 0}$ for any $((i,j); (r,s))\in \mathcal{C}^-$.
\item[$\bullet$] For any $1\leq k\leq n-1$ we have, $l_{ki}-l_{kj}\in \mathbb{Z} $ if and only if $(k,i)$ and $(k,j)$ are in the same connected component of $G(\mathcal{C})$.
\end{itemize}
\item[(ii)] If $[l]$ satisfies $\mathcal{C}$ then ${\mathcal B}_{\mathcal{C}}([l])$ will be  the set of all 
$[l+z]$, where $z\in {\mathbb Z}^{\frac{n(n+1)}{2}}$ satisfying $\mathcal{C}$ with $z_{ni}=0$ for $1\leq i\leq n$. By $V_{\mathcal{C}}([l])$ we denote the complex vector space spanned by ${\mathcal B}_{\mathcal{C}}([l])$.
\end{itemize}
\end{definition}

We call $\mathcal{C}$ \emph{noncritical} if $l_{ki}\neq l_{kj}$, $1\leq i< j\leq k\leq n-1$, for all $[l]$ that satisfies $\mathcal{C}$.

 Following \cite{FRZ}, Definition 4.18, 
 we call a set of relations $\mathcal{C}$ \emph{reduced} if for every $(k,j)\in \mathfrak{V}(\mathcal{C})$ the following conditions
are satisfied:
\begin{itemize}
\item[(i)] There exists at most one $i$ such that $((k,j);(k+1,i))\in \mathcal{C}$;
\item[(ii)] There exists at most one $i$ such that $((k+1,i);(k,j))\in\mathcal{C}$;
\item[(iii)] There exists at most one $i$ such that $((k,j);(k-1,i))\in\mathcal{C}$;
\item[(iv)] There exists at most one $i$ such that $((k-1,i);(k,j))\in \mathcal{C}$;
\item[(v)] No relations in the top row follow from other relations.
\end{itemize}

\

Let $\mathcal{C}$ be a set of relations. For $i<j$ we say that $(k,i)$ and $(k,j)$ form an adjoining pair if they are in the same indecomposable subset of $\mathcal{C}$ and there is no pair $(k,s)$ such that
$(k,i)\succeq_{\mathcal{C}}(k,s)\succeq_{\mathcal{C}}(k,j)$.

Denote by $\mathfrak{F}$ the set of all indecomposable sets of relations $\mathcal{C}$ satisfying the following condition:

\begin{itemize}
\item[(i)] $\mathcal{C} $ is noncritical and reduced;
\item[(ii)]  $\mathcal{C} $ has no subsets of the form { $\{((k,i); (k+1,t)),\ ((k+1,s);(k,j))\}$   }with $i<j$ and $s<t$;
\item[(iii)] If $(n,i)$ and $(n,j)$, $i\neq j$, are in the same indecomposable subset of $\mathcal{C}$, then $(n,i)\succeq_{\mathcal{C}}(n,j)$ or $(n,j)\succeq_{\mathcal{C}}(n,i)$.
\item[(iv)] For every adjoining pair $(k,i)$ and $(k,j)$, $1\leq k\leq n-1$, there exist $p, q$ such that $\mathcal{C}_{1}\subseteq\mathcal{C}$
or, there exist $s<t$ such that $\mathcal{C}_{2}\subseteq\mathcal{C}$, where the graphs associated to $\mathcal{C}_{1}$ and $\mathcal{C}_{2}$ are as follows
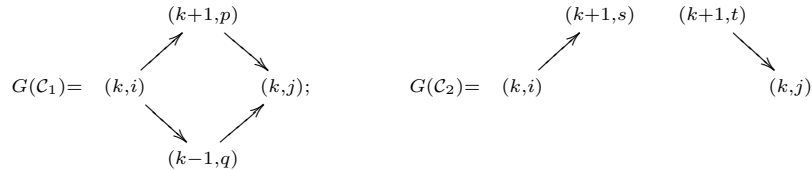
\begin{center}
\begin{tabular}{c c c c}
\xymatrixrowsep{0.5cm}
\xymatrixcolsep{0.1cm}
\xymatrix @C=0.2em{
  &   &\scriptstyle{(k+1,p)}\ar[rd]   &   & \\
 \scriptstyle{G(\mathcal{C}_{1})=}  &\scriptstyle{(k,i)}\ar[rd] \ar[ru]  &    &\scriptstyle{(k,j)};   &  \\
   &   &\scriptstyle{(k-1,q)}\ar[ru]   &   & }
&\ &
\xymatrixrowsep{0.5cm}
\xymatrixcolsep{0.1cm}\xymatrix @C=0.2em {
   &   &\scriptstyle{(k+1,s)}    &   &\scriptstyle{(k+1,t)}\ar[rd]&& \\
  \scriptstyle{G(\mathcal{C}_{2})=} &\scriptstyle{(k,i)} \ar[ru]  & &   & & \scriptstyle{(k,j)} \\
   &   &   &   & &&}
\end{tabular}
\end{center}
\end{itemize}

Disconnected unions of sets of $\mathfrak{F}$ are called \emph{admissible}. We have

\begin{theorem} \cite{FRZ}\label{thm-gln}
Let $\mathcal{C}$ be an   admissible set of relations. Then for any
$[l^0]$ satisfying $\mathcal{C}$,  $V_{\mathcal{C}}([l^0])$ has a $\gl(n)$-module structure and
the action of the generators of $\gl_n$ on any basis tableau $[l]$ is given by
\begin{equation}
\begin{split}
&E_{k,k+1}[l]=-\sum_{j=1}^{k}\frac{\prod_{i} (l_{k+ 1,i}-l_{k,j})}{\prod_{i\neq j}(l_{k,i}-l_{k,j})}
[l+\delta^{kj}],\\
&E_{k+1,k}[l]=\sum_{j=1}^{k}\frac{\prod_{i} (l_{k-1,i}-l_{k,j})}{\prod_{i\neq j} (l_{k,i}-l_{k,j})}[l-\delta^{kj}],\\
&E_{kk}[l]=\left(\sum_{i=1}^{k}l_{k,i}-\sum_{i=1}^{k-1}l_{k-1,i}+k-1\right)[l]. \\
\end{split}
\end{equation}
Here $[l\pm\delta^{kj}]$ is the tableau obtained
from $[l]$ by adding $\pm 1$ to the $(k,j)$th entry of $[l]$; if
 $[l]$  is not in ${\mathcal B}_{\mathcal{C}}([l^0])$ then it is assumed to be zero.
\end{theorem}

Module  $V_{\mathcal{C}}([l^0])$ in Theorem \ref{thm-gln} is a \emph{relation} $\gl(n)$-module.  We will extend the construction of relation modules to the Lie superalgebra $\gl(m|n)$ in the next section.

\

\subsection{Relation modules for $\gl(m|n)$}
\begin{definition}
The set $\mathcal C=(\mathcal{C}_1,\mathcal{C}_2)$ is called \emph{admissible} if $\mathcal{C}_1$
 and $\mathcal{C}_2$ are admissible sets of relations of $\gl(m)$ and $\gl(n)$ respectively.
\end{definition}

\begin{definition}Let $\mathcal C=(\mathcal{C}_1,\mathcal{C}_2)$ be admissible.
\begin{itemize}
\item[(i)] We will say that a tableau $[l]$ of $\gl(m|n)$ satisfies $\mathcal{C}_1$ if:
\begin{itemize}
\item[$\bullet$] $l_{ij}-l_{rs}\in \mathbb{Z}_{\geq 0}$ for any $((i,j); (r,s))\in \mathcal{C}_1^+$;
\item[$\bullet$] $l_{ij}-l_{rs}\in \mathbb{Z}_{> 0}$ for any $((i,j); (r,s))\in \mathcal{C}_1^-\cup\mathcal{C}_1^0$;
\item[$\bullet$] For any $1\leq k\leq n$ we have, $l_{ki}-l_{kj}\in \mathbb{Z} $ if and only if $(k,i)$ and $(k,j)$ are in the same connected component of $G(\mathcal{C}_1)$.
\end{itemize}
\item[(ii)] We will say that $[l]$ of $\gl(m|n)$ satisfies $\mathcal{C}_2$ if:
\begin{itemize}
\item[$\bullet$] $l_{ij}-l_{rs}\in \mathbb{Z}_{\leq 0}$ for any $((i,j); (r,s))\in \mathcal{C}_2^+\cup\mathcal{C}_2^0$;
\item[$\bullet$] $l_{ij}-l_{rs}\in \mathbb{Z}_{< 0}$ for any $((i,j); (r,s))\in \mathcal{C}_2^-$;
\item[$\bullet$] For any $m+1\leq k\leq m+n$ we have, $l_{ki}-l_{kj}\in \mathbb{Z} $ if and only if $(k,i)$ and $(k,j)$ are in the same connected component of $G(\mathcal{C}_2)$.
\end{itemize}
\item[(iii)] 
We say that a tableau $[l]$ of $\gl(m|n)$ satisfies $\mathcal{C}$ if:
\begin{enumerate}
\item $[l]$ satisfies $\mathcal C_1$ and $\mathcal C_2$;
\item  $\l_{ki}-\l_{k-1,i}\equiv\theta_{k-1,i}(l)\in\{0,1\}, 1\leq i\leq m, m+1\leq k\leq m+n$;
\item  $\l_{ki}- \l_{kj}\neq 0$ for $m+1\leq k\leq m+n-1$, $1\leq i\leq m<j \leq k$;
\item If $[l']$ satisfies conditions (1)-(2) and
 ${l'}_{ki}={l}_{ki}+z_{ki}$, where  $z_{ki}\in \mathbb Z$, $z_{m+n,i}=0$ for all $i$,
 then $[l']$ satisfies condition (3).
\end{enumerate}
\end{itemize}
\end{definition}
\
Let 
$[l^0]=({l^0}_{ij})$ be a tableau satisfying $\mathcal C$ and ${\mathcal B}_{\mathcal{C}}([l^0])$ the set of all possible tableaux $[l]=(l_{ij})$
such that
\begin{enumerate}
\item[(1)] ${l}_{ki}={l^0}_{ki}+z_{ki}$, $z_{ki}\in \mathbb Z$, $1\leq k\leq m+n-1$, $1\leq i\leq k$;
\item[(2)] [l] satisfies $\mathcal C$.
\end{enumerate}

Let $V_{\mathcal C}([l^0])$ be the vector space spanned by ${\mathcal B}_{\mathcal{C}}([l^0])$.

 We are now in the position to state the main result of this section.

\begin{theorem}\label{admissible}
Let $\mathcal C=(\mathcal{C}_1,\mathcal{C}_2)$ be an admissible set of relations for $\gl(m|n)$,
$[l^0]$  a tableau satisfying $\mathcal C$.
Then the formulas \eqref{action h}-\eqref{action f} define a $\gl(m|n)$-module structure on $V_{\mathcal C}([l^0])$.
\end{theorem}

\

The module 
$V_{\mathcal C}([l^0])$ is called \emph{quasi typical} $\gl(m|n)$-module. Note that any essentially typical module is quasi typical. Also note that quasi typical module need not be 
typical. The module $V_{\mathcal C}([l^0])$ is infinite dimensional in general, it may or may not be highest weight module.

The formulas \eqref{action h}-\eqref{action f} define a 
$\hat{\mathfrak{g}}$-module structure on the space
 $V_{\mathcal C}(l^0)$.
Recall that we have a homomorphism 
$\epsilon: \hat{\mathfrak{g}}\rightarrow \gl(m|n)$ given by
$e_i\mapsto E_{i,i+1}, f_i\mapsto E_{i+1,i},1\leq i\leq m+n-1 $ and $h_j\mapsto E_{jj},1\leq j\leq m+n$.  Denote the kernel of the homomorphism $\epsilon$
by $K$.

\

\subsection{Relation removal method}
Let  $\mathcal C=(\mathcal{C}_1,\mathcal{C}_2)$. Suppose $(k,i)$ is a vertex involved in the set of relations $\mathcal C$, that is $(k,i)$ is the starting or the ending vertex of an arrow in the graph of $\mathcal C$.
We call a vertex $(k,i)$ in the graph of  $\mathcal C_1$(resp. $\mathcal C_2$) \emph{maximal} if there is no $(s,t)$ such that $((s,t);(k,i))\in \mathcal C_1 $(resp. $((k,i);(s,t))\in \mathcal C_2$). A  \emph{minimal} pair is defined similarly.
Suppose $\mathcal{C}$ is  admissible set  and  $(k,i)$ is maximal or minimal. Denote by $ \mathcal{C}_{ki} $ the set of relations obtained from $\mathcal{C}$ by removing all relations that involve $(k,i)$. Following   \cite{FRZ} we will say that $ \mathcal{C}_{ki}\subsetneq\mathcal{C}$ is obtained from $\mathcal{C}$ by the \emph{relation removal} method, or the RR-method for short.

\begin{theorem}\label{rr}
Let $ \mathcal{C}=(\mathcal{C}_1,\mathcal{C}_2)$ be an admissible set of relations for which Theorem \ref{admissible} holds. If $\mathcal{C}'$ is obtained from $ \mathcal{C}$ by the RR-method then $\mathcal{C}'$ is admissible and Theorem \ref{admissible} holds for $\mathcal{C}'$.
\end{theorem}

\begin{proof} Suppose that $\mathcal{C}'$ is obtained from $\mathcal{C}$ by removing $(i,j)$. Then clearly $\mathcal{C}'$ is admissible. 
To show that Theorem \ref{admissible} holds for $\mathcal{C}'$ we need to prove that for any  generator $g\in K$ and for any  tableau $[l']$ satisfying $\mathcal C$ one has  $g[l']=0$.   

Consider a tableau $[l]$  satisfying $\mathcal{C}$ and the following conditions. If $i\neq m$ then   $l_{rt}=l'_{rt}$ if $(r, t)\neq (i, j)$. If $i= m$ then $l_{rt}=l'_{rt}$ for $(r, t)\neq (k, j), m\leq k\leq m+n$
and $\theta(l)=\theta(l')$ for $m\leq k\leq m+n-1$. Also define 
$\delta=\delta_{ij}$ if $i\neq m$ and 
 $\delta=\sum_{k=m}^{m+n}\delta_{kj}$ if $i= m$.

Let $s$ be positive (respectively, negative) integer if $(i,j)$ is maximal (respectively minimal) with $|s|>>0$.
Expand $g[l+ s\delta]$ and $g[l']$  step by step (using formulas \eqref{action h}-\eqref{action f}). We have
\begin{align*}
g[l+ s\delta]=\sum\limits_{z\in A}g_{z}(l+ s\delta)[l+ s\delta+z],
\end{align*}
where $A$ is the set of $z$ such that $[l+ s\delta+z]\in {\mathcal B}_{\mathcal{C}}([l])$.
 Suppose that the tableau $[l+ s\delta+z]$ appears in the expansion of 
$g[l+ s\delta]$.   Then  the tableau $[l'+z]$  appears in the expansion of 
$g[l']$. Moreover,   $[l+ s\delta+z]$  belongs to ${\mathcal B}_{\mathcal{C}}([l])$
if and only if  $[l'+z]$   belongs to
${\mathcal B}_{\mathcal{C}'}([l'])$. Hence we have
\begin{align*}
g[l']=\sum\limits_{z\in A}g_{z}(l')[l'+z]. 
\end{align*}

Since  for which Theorem \ref{admissible} holds for $\mathcal C$, one has that $g[l+ s\delta]=0$
and the values of rational functions $g_{z}(l+ s\delta)$ are $0$ for infinitely many values of $s$ for all $z\in A$. Thus
$g_{z}(l')=0$ for all $z\in A$ and  hence Theorem \ref{admissible} holds
for $\mathcal{C}'$.
\end{proof}

\

\subsection{Proof of Theorem \ref{admissible}}

 Now we prove  Theorem \ref{admissible}. The proof follows the idea of  the proof of \cite{FRZ}, Theorem 4.33 for $\gl(n)$. We provide the details for the sake of 
 completeness. 
 
Let $\mathcal C$ be an admissible set of relations and  $[l^0]$ is a tableau  satisfying $\mathcal{C}$.     Our goal is to show that $K$ vanishes on  $[l^0]$. 

We  say that a tableau $[l]$ satisfies the \emph{generic} condition if 
\begin{equation}\label{generic condition}
l_{ki}-l_{kj}\notin \mathbb Z  \text{ for }  1\leq i\leq m, m+1\leq j\leq k\leq m+n.
\end{equation}

\begin{lemma}\label{generic}
Let $g\in U(\hat{\mathfrak{g}})$, $[l^0]$ is a tableau  satisfying $\mathcal{C}$. If $g[l]=0$ for any tableau $[l]$
satisfying  $\mathcal{C}$ and the generic condition then $g[l^0]=0$.
\end{lemma}

\begin{proof}
We have  $$g[l]=\sum_{z\in A}g_z(l)[l+z],$$ for some finite set of integers $A$. Here $g_z$ is some rational function and $g_z(l)$ is its evaluation in $l$. Let $1_n$ be the tableau such that the $(k,j)$th entry is $1$
for $m+1\leq j\leq k\leq m+n$ and all other entries are $0$.
 There exist infinitely many $x\in\mathbb C$ such that $[l^0+x1_n]$ satisfies the generic condition \eqref{generic condition}, 
 and for each such $x$ one has 
 $$g[l^0+x1_n]=\sum_{z\in A}g_z(l^0+x1_n)[l+z]=0$$ by the assumption. Hence $g_z(l^0+x1_n)=0$ for infinitely many $x$. 
We conclude that $g_z(l^0)=0$ and the statement follows.
\end{proof}

By Lemma \ref{generic} it is sufficient  
 to prove the statement  of the theorem for 
$[l]$ with the
generic condition \eqref{generic condition}. Let $g$  be a generator of $K$.  Then
$$g[l]=\sum_{z}\varphi_z(l)[l+z].$$ 
Our goal is to verify that  $g[l]=0$.  To do so we consider every pair of indices $(k,i)$ involved in  $\mathcal C$ with $z_{ki}\neq 0$ and construct an admissible  
set of relations $\mathcal C'$ for which Theorem \ref{admissible} holds.  
Let $[v]$ be a tableau with some variable entries that satisfies $\mathcal C'$ (hence $g[v]=0$). Then 
$\varphi_z(l)$ is equal to the evaluation of the coefficient $\varphi_z(v)$ of $[v+z]$ in 
$g[v]$ at $[l]$ which is $0$.
This proves that $\varphi_z(l)=0$.
Note that for same $g$ and different $z$ one  chooses different sets of relations $\mathcal C'$
and  tableaux $[v]$.  In the following we provide the sets of relations $\mathcal C'$ and the  tableaux $[v]$ for each generator $g\in K$.

Define
\begin{align}
\mathfrak V_m=\{(i,j)|1\leq j\leq i\leq m\},\\
\mathfrak V_n=\{(i,j)|m+1\leq j\leq i\leq m+n\},\\
\mathfrak V=\{(i,j)|m+1\leq i\leq m+n-1,1\leq j\leq m\}.
\end{align}

Let $v$ be the variable tableau such that $v_{ab}$ are variables for $(a,b)\in \mathfrak V_m\cup\mathfrak V_n$ and $\theta_{ab}(v)=\theta_{ab}(l)$ for $(a,b)\in \mathfrak V$.

\begin{itemize}
\item [(i)] $[h_{i},h_{j}][l]=0$.
Let $g=[h_{i},h_{j}]$. Then 
$$g[l]=\varphi(l)[l]=0.$$
 Then  $\varphi(l)$ is the evaluation of the coefficient of 
$[v]$ in $g[v]$ at $[l]$ which is $0$. Thus $[h_{i},h_{j}][l]=0$.
\\
\item [(ii)] $[h_{i},e_{j}][l]=(\delta_{ij}-\delta_{i,j+1})e_{j}[l].$
Denote $g=[h_{i},e_{j}]-(\delta_{ij}-\delta_{i,j+1})$. 
Then $$g[l]=\sum_{s=1}^{j}\varphi_{s}(l)[l+\delta_{js}].$$
For nonzero vector $[l+\delta_{js}]$,  the coefficient $\varphi_{s}(l)$ is equal to the evaluation of the coefficient of $[v+\delta_{js}]$ in $g[v]$ at $[l]$ which is $0$.
Therefore,  $[h_{i},e_{j}][l]=(\delta_{ij}-\delta_{i,j+1})e_{j}[l]$.
\\
\item [(iii)] $[h_{i},f_{j}][l]=-(\delta_{ij}-\delta_{i,j+1})f_{j}[l].$
The proof is similar to (ii).
\\
\item [(iv)]$[e_{i},f_{j}][l]=0  \text{ if }  i\neq j$.
Denote $g=[e_{i},f_{j}]$. Then 
$$g[l]=\sum_{s=1}^{i}\sum_{t=1}^{j}\varphi_{s,t}(l)[l+\delta_{is}-\delta_{jt}].$$
\begin{itemize}
\item[(a)] Suppose there is no relation between $(i,s)$ and $(j,t)$.
Let $[v]$ be the tableau with variable entries. 
\item[(b)] Suppose there is a relation between $(i,s)$ and $(j,t)$. Then it is $((i,s);(j,t))$
or  $((j,t);(i,s))$ and $j=i\pm 1$. We set $\mathcal C'=\{((i,s);(j,t))\}$
or  $\mathcal C'=\{((j,t);(i,s))\}$ respectively. Then $C'$ is admissible. Moreover, 
applying Theorem \ref{rr} to the "standard" set of relations (satisfied by all essentially typical tableaux), we obtain that $C'$ satisfies Theorem \ref{admissible}. 
Let $v_{is}=l_{is},v_{jt}=l_{jt}$ and all other entries of $[v]$ are variables. 
\end{itemize}

In both cases above, the coefficient $\varphi_{s,t}(l)$ by the tableau $[l+\delta_{is}-\delta_{jt}]$ is the evaluation of the coefficient 
$\varphi_{s,t}(v)$ of $[v+\delta_{is}-\delta_{jt}]$ in $g[v]$ at $[l]$ which is $0$. Hence $\varphi_{s,t}(l)$ and 
 $[e_{i},f_{j}][l]=0$.

\item [(v)]$[e_{i},f_{i}][l]=(h_{i}-h_{i+1})[l] \text{ for } i\neq m$.
If $i<m$, the desired equality follows from \cite{FRZ}, Theorem 4.33. Suppose $i>m$. 
Denote $g=[e_{i},f_{i}]-h_{i}+h_{i+1}$. Then 
$$g[l]=\sum_{s\neq t}\varphi_{s,t}(l)[l+\delta_{is}-\delta_{it}]
+\varphi(l)[l].$$
Since there is no relation between $(i,s)$ and $(i,t)$, it can be verified that
$\varphi_{s,t}(l)=0$ using the same argument as in (a) of (iv).
Define operators $\hat{e}_i$ and $\hat{f}_i$ by
\begin{equation}
\hat e_{i}[l]=
-\sum_{s=m+1}^{i}  \Pi_{t=1}^{m}\left(\frac{l_{it}-l_{is}+1}{l_{i-1,t}-l_{is}+1} \right)
  \frac{\Pi_{t=m+1}^{i+1}(l_{i+1,t}-l_{is})}
  {\Pi_{t\neq s=m+1}^{k} (l_{it}-l_{is})}[l+\delta_{is}] ,
\end{equation}
\begin{equation}
\hat f_{i}[l]=
\sum_{s=m+1}^{i}
\prod_{t=1}^{m}\left(\frac{l_{it}-l_{is}-1}{l_{i+1,t}-l_{is}-1}\right)
\frac{\prod_{t=m+1}^{i-1}(l_{i-1,t}-l_{is})}{\prod_{t\neq s =m+1}^{i}(l_{it}-l_{is})}
[l-\delta_{is}].
\end{equation}
Since the $\theta$ is $0$ or $1$, $\varphi(l)$ is equal to the coefficient of $[l]$ in $(\hat e_i \hat f_i-\hat f_i \hat e_i-h_{i}+h_{i+1})[l]$.
Since the set $\mathcal C'=(\emptyset, \emptyset)$ is admissible by Theorem \ref{rr}, then the coefficient of
$[v]$ in $g[v]$ is $0$. Thus the coefficient of $[v]$ in $(\hat e_i \hat f_i-\hat f_i \hat e_i-h_{i}+h_{i+1})[v]$ is $0$. Using the argument of 
\cite{FRZ}, Lemma 4.31,  we have that the coefficient of $[l]$ in $(\hat e_i \hat f_i-\hat f_i \hat e_i-h_{i}+h_{i+1})[l]$ is the limit of the coefficient
 of $[v]$ in $(\hat e_i \hat f_i-\hat f_i \hat e_i-h_{i}+h_{i+1})[v]$ by taking $v$ to $l$.
Therefore, $\varphi(l)=0$.
\item [(vi)]$[e_{m},f_{m}][l]=(h_{m}+h_{m+1})[l]$.
The proof is similar to the case (v).

\item[(vii)] 
$[e_{i},e_{j}][l]=[f_{i},f_{j}][l]=0 \text{ if } |i-j|> 1$.
The equality can be obtained by the same argument as in (a) of (iv).

\item[(viii)] $[e_m,e_m][l]=[f_m,f_m][l]=0.$ It can be proved using the same argument as in (vi).

Denote $g=[e_m,e_m]$. Then 
$$g[l]=\sum_{s\leq t}\varphi_{s,t}(l)[l+\delta_{ms}+\delta_{mt}].$$

If $l_{ms}-l_{mt}>1$ or $s=t$, then $\varphi_{s,t}(l)$
is equal to the evaluation of the coefficient of $[v+\delta_{is}+\delta_{it}]$ in $g[v]$ at $[l]$ which is $0$.
If $l_{ms}-l_{mt}=1$ then  without loss of generality we can assume that there exists $p$ such that
$((m,s);(m-1,p)),((m-1,p);(m,t))\in \mathcal C$. Thus $l_{m-1,p}=l_{m,s}$ and
$[l+\delta_{ms}+\delta_{mt}]$ does not satisfy $\mathcal C$. This proves $[e_m,e_m][l]=0$.
Using the same arguments one has that $[f_m,f_m][l]=0$.

\item [(ix)] 
$[e_i,[e_i,e_{i\pm 1}]][l]=[f_i,[f_i,f_{i\pm 1}]][l]=0, \text{ for } i\neq m$. The equalities follow from \cite{FRZ} for $i<m$. We show the equalities for $i>m$.
Denote $g=[e_i,[e_i,e_{i\pm 1}]]$. Then

$$g[l]=\sum_{r\leq s}\sum_{t}\varphi_{r,s,t}(l)[l+\delta_{ir}+\delta_{is}+\delta_{i\pm 1,t}].$$

\begin{itemize}
\item[(a)] Suppose that $r,s,t\geq m+1$. 
\begin{itemize}
\item[(1)]
If there is no relation between $(i,r)$, $(i,s)$ and $(i\pm 1,t)$ then take 
 $\mathcal C'=(\emptyset,\emptyset)$ and the variable tableau $[v]$.
\item[(2)] Suppose one of $\{(i,r), (i,s)\}$ is connected to $(i\pm 1,t)$.
Let $\mathcal C'$ be the set consisting of this relation. 
Then $\mathcal C'$ is admissible and satisfies Theorem \ref{admissible}.
Without loss of generality we may assume that $\mathcal C'=\{((i,r);(i\pm 1, t))\}$.
Choose $[v]$ with 
 $v_{ir}=l_{ir}$ and
$v_{i\pm 1,t}=l_{i\pm 1,t}$.  Applying the argument as above in both cases (1) and (2) we obtain $g[l]=0$.

\item[(3)] Suppose there are relations between $(i+ 1, t)$ and both $(i,r)$ and $(i,s)$. Without loss of generality we may assume that $r<s$.
Then $\{((i,r);(i+1,t)),((i+1,t);(i,s))\}\subseteq \mathcal{C}$ and there exists $p$
 such that  $\{((i,r);(i- 1,p)),((i- 1,p);(i,s))\}\subseteq \mathcal{C}$.
Set $\mathcal{C}'=$ $\{((i,r);(i+ 1, t)),((i+ 1 ,t);(i,s)),((i,r);(i- 1,p)),((i-1,p);(i,s))\}.$ Then $\mathcal C'$ is admissible with the
 associated graph  as follows.
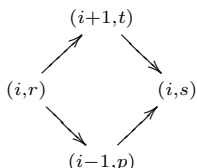
\begin{center}
\begin{tabular}{c c c c}
\xymatrixrowsep{0.5cm}
\xymatrixcolsep{0.1cm}
\xymatrix @C=0.2em{
  &   &\scriptstyle{(i+1,t)}\ar[rd]   &   & \\
 &\scriptstyle{(i,r)}\ar[rd] \ar[ru]  &    &\scriptstyle{(i,s)}  &  \\
   &   &\scriptstyle{(i-1,p)}\ar[ru]   &   & }
\end{tabular}
\end{center}
Hence $\mathcal C'$ satisfies Theorem \ref{admissible}.
We choose $[v]$ such that
 $v_{ir}=l_{ir}$, $v_{is}=l_{is}$,
$v_{i-1,p}=l_{i-1,p}$, $v_{i+1,t}=l_{i+1,t}$.
 
The coefficient $\varphi_{r,s,t}(l)$ of the nonzero vector $[l+\delta_{ir}+\delta_{is}+\delta_{i+ 1,t}]$ is the evaluation at $[l]$ of the coefficient of $[v+\delta_{ir}+\delta_{is}+\delta_{i+1,t}]$ in $g[v]$,  which is $0$. Hence $\varphi_{r,s,t}(l)=0$. 

\item[(4)]
Suppose there are relations between $(i-1, t)$ and both $(i,r)$ and $(i,s)$. Without loss of generality we may assume that $r<s$.
Then, either there exists $p$
such that  $\{((i,r);(i- 1,p)),((i- 1,p);(i,s))\}\subseteq \mathcal{C}$,
or there exist $p_k < q_k$, $i\leq k \leq m+n$ such that $p_{i}=r,q_{i}=s$ and
$((i_k,p_k);(i_{k+1},p_{k+1})), ((i_k,q_k);(i_{k+1},q_{k+1}))$ $\in \mathcal C$ for $i+1\leq k\leq m+n-1$.
In the first case we choose
 $\mathcal{C}'=\{((i,r);(i- 1, t)),((i- 1 ,t);(i,s)),((i,r);(i+ 1,p)),((i+ 1,p);(i,s))\}$  and tableau $[v]$ such that
 $v_{ir}=l_{ir}$, $v_{is}=l_{is}$,
$v_{i-1,p}=l_{i-1,p}$, $v_{i+1,t}=l_{i+1,t}$. 
In the second case
we choose $\mathcal{C}'=\{((i,r);(i- 1 ,t)),((i-1, t);(i,s)\}\cup \{((m+n,p_{m+n});(m+n,q_{m+n}))\} \cup \{((i_k,p_k);(i_{k+1},p_{k+1}))|i\leq k\leq m+n-1\}\cup\{((i_k,q_{k+1});(i_{k+1},q_{k}))|i\leq k\leq m+n-1,\} $ and
tableau $[v]$ such that $v_{i-1,t}=l_{i-1,t}$,
$v_{k,p_k}=l_{i_k,p_k}$, $v_{k,q_k}=l_{k,q_k}$, $i\leq k\leq m+n$.
The associated graphs are  respectively as follows.

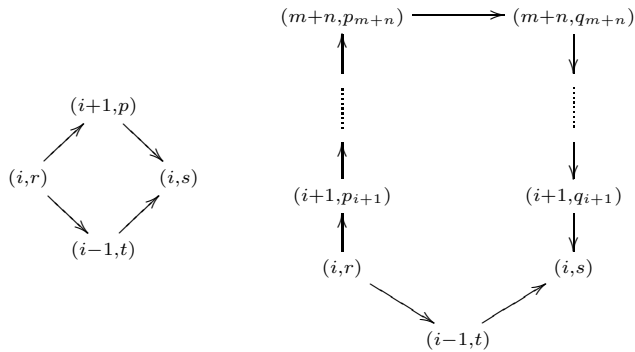
\begin{center}
\begin{tabular}{c c c c}
\xymatrixrowsep{0.5cm}
\xymatrixcolsep{0.1cm}
\xymatrix @C=0.2em{
  &   &\scriptstyle{(i+1,p)}\ar[rd]   &   & \\
 &\scriptstyle{(i,r)}\ar[rd] \ar[ru]  &    &\scriptstyle{(i,s)}  &  \\
   &   &\scriptstyle{(i-1,t)}\ar[ru]   &   & }
\end{tabular}
\begin{tabular}{c c c}
\xymatrixrowsep{0.5cm}
\xymatrixcolsep{0.1cm}\xymatrix @C=0.2em 
{
 \scriptstyle{(m+n,p_{m+n})} \ar[rr] & &\scriptstyle{(m+n,q_{m+n})} \ar[d]    & \\
  \ar[u] && \ar@{.}[d]   \\
      \ar@{.}[u]  &    &    \ar[d] \\
  \scriptstyle{(i+1,p_{i+1})}\ar[u]  &  &  \scriptstyle{(i+1,q_{i+1})}\ar[d] \\
\scriptstyle{(i,r)} \ar[u] \ar[dr]   &   & \scriptstyle{(i,s)} \\
& \scriptstyle{(i-1,t)}\ar[ur] &&
}
\end{tabular}
\end{center}

In each of these cases the set $\mathcal C'$ is admissible and it satisfies Theorem \ref{admissible}. 
The coefficient $\varphi_{r,s,t}(l)$ of the nonzero vector $[l+\delta_{ir}+\delta_{is}+\delta_{i- 1,t}]$ is the evaluation at $[l]$ of the coefficient of $[v+\delta_{ir}+\delta_{is}+\delta_{i- 1,t}]$ in $g[v]$  which equals $0$. Hence $\varphi_{r,s,t}(l)=0$. 

\end{itemize}

\item[(b)] Suppose that only one number in $\{r,s,t\}$ is $\leq m$.
Then there exists at most one relation between $(i,r), (i,s)$ and $(i\pm 1,t)$.
By the above argument, $\varphi_{r,s,t}(l)=0$.
\item[(c)] Suppose that two numbers in $\{r,s,t\}$ are $\geq m$. Similarly to (viii),
$\varphi_{r,s,t}(\Lambda)=0$.
\end{itemize}
The equation $[f_i,[f_i,f_{i\pm 1}]][l]=0$ can be shown by a similar argument.

\item [(x)] $[e_m[e_{m\pm1},[e_m,e_{m\mp1}]]][l]=[f_m[f_{m\pm 1},[f_m,f_{m\mp1}]]][l]=0.$
Denote \\ $g=[e_m[e_{m\pm1},[e_m,e_{m\mp1}]]]$. 
Let $\mathcal C'=(\emptyset, \emptyset)$ be the empty set of relations and $v=(v_{ij})$ the  tableau with variable entries.
Applying the Gelfand-Tsetlin formulas   to a nonzero $x\in \g$ one can write $x[v]$  as follows:
$$x[v]=\sum_{z}a(v)[v+z]+\sum_{z'}a'(v)[v+z'],$$
where $[l+z]$ satisfies $\mathcal C$ and 
$[l+z']$ does not satisfy $\mathcal C$.
Hence $g[v]$ has the form:
\begin{equation*}
g[v]=\sum
(\varphi_{s,t,p,q}(v)+\varphi'_{s,t,p,q}(v))
[v+\delta_{ms}+\delta_{mt}+\delta_{m+ 1,p}+\delta_{m-1,q}].
\end{equation*}
Since $\mathcal C'$ is admissible, $\varphi_{s,t,p,q}(v)+\varphi'_{s,t,p,q}(v)=0$.
Evaluating at $[l]$ we have
$$g[l]=\sum_{s\leq t}\sum_{p,q}
\varphi_{s,t,p,q}(l)
[l+\delta_{ms}+\delta_{mt}+\delta_{m+ 1,p}+\delta_{m-1,q}].$$

If $[l+\delta_{ms}+\delta_{mt}+\delta_{m+ 1,p}+\delta_{m-1,q}]$ does not satify $\mathcal C$, then $\varphi_{s,t,p,q}(l)$ is zero by definition.
So we can assume that  $[l+\delta_{ms}+\delta_{mt}+\delta_{m+ 1,p}+\delta_{m-1,q}]$  satisfies $\mathcal C$.
If $l_{mi}-l_{mj}\neq \mathbb Z$ or $l_{mi}-l_{mj}\neq \mathbb Z$, then
all the tableaux appearing in the expansion of $g[l]$ satisfy $\mathcal C$. So we do not have $\varphi'_{s,t,p,q}(v)$ in this case and $\varphi_{s,t,p,q}(l)=0$.
Suppose $l_{mi}-l_{mj}=1$. The rational function $\varphi'_{s,t,p,q}$ is regular at $v_{mi}=v_{mj}-1$. Moreover, its evaluation at $[l]$ is zero.
Thus  $\varphi_{s,t,p,q}(l)=0$.

Using the same argument, we have $[f_m[f_{m\pm 1},[f_m,f_{m\mp1}]]][l]=0$.

\end{itemize}

This completes the proof of Theorem \ref{admissible}.

\

\begin{example}[Highest weight modules]
Let $\lambda=(\lambda_1,\ldots ,\lambda_{m+n})$ such that $\lambda_{i}-\lambda_j\notin \mathbb{Z}$ or $\lambda_{i}-\lambda_j> i-j$ for $1\leq i<j\leq m$, $m+1\leq i<j\leq m+n$ and
$\la_s+\la_t\notin \mathbb Z$ for $1\leq s\leq m<t\leq m+n$.
Let $[l^0]$ be the tableau such that
 $$l_{ki}^0=\la_{ki}-i+1, (1\leq i\leq m); \quad l_{kj}^0=-\la_{kj}+j-2m, (m+1\leq j\leq k),$$
 where $\lambda_{ki}=\lambda_i$, $i=1, \ldots, m+n$. Let 
 $\mathcal{C}_1$ and $\mathcal{C}_2$ be maximal sets of relations satisfied by $[l^0]$.
Then the irreducible highest weight module $L(\lambda)$ is isomorphic to $V_{\mathcal{C} }([l^0])$.

\end{example}

\

\section{Irreducibility of quasi typical modules}
In this section we obtain necessary and sufficient conditions for the irreducibility of  quasi typical $\g$-modules.
We have

\begin{proposition}\label{prop action of gt subalg}
The action of $B(t)$ on any basis tableau $[l]$ of  $V_{\mathcal C}([l^0])$ is given by
\begin{equation}\label{action of gt subalg}
B_k(t)[l]=
\left\{ 
\begin{aligned}
&(1+tl_{k1})\cdots(1+tl_{kk})[l],      &\text{if } 1\leq k\leq m,\\
&\frac{(1+tl_{k1})\cdots(1+tl_{km})}
{(1+tl_{k,m+1})\cdots(1+tl_{k,k})} [l],    &\text{if } m+1\leq k\leq m+n.\\
\end{aligned}
 \right.
\end{equation}
\end{proposition}

\begin{proof}
Let $ \mathcal{C}$ be an admissible set of relations. If
$\tilde{\mathcal{C}}$ is obtained from  $ \mathcal{C}$ by the RR-method
 then  $B_k (t)$ acts on   $V_{\mathcal{C}}([l^0])$   by
  \eqref{action of gt subalg}  for any $[l^0]$ if and only if
it acts by  \eqref{action of gt subalg}  on   $V_{\tilde{\mathcal{C}}}([\tilde{l^0}])$ for any $[\tilde{l^0}]$. 
This follows by the argument from   \cite{FRZ},  Lemma 5.1.
Since the set $\emptyset$ can be obtained from any admissible set by the RR-method, the statement follows.
\end{proof}

\begin{corollary}\label{corollary separation}
Let $\sum_{i=1}^{r}a_{i}[l_i]\in V_{\mathcal{C}}([l])$ with all $a_{i}$ non-zero, and $[l_i]$ are mutually distinct.
Then   $[l_i]\in V_{\mathcal{C}}([l])$ for all $1\leq i\leq r$.
\end{corollary}
\begin{proof}
It follows from Proposition \ref{prop action of gt subalg} that
 the Gelfand-Tsetlin subalgebra has different characters
 on tableaux $[l]\neq [l']$. Then  $[l_i]\in V_{\mathcal{C}}([l])$ for all $1\leq i\leq r$.
\end{proof}

\

\begin{theorem}\label{irreducibility} Let $\mathcal C$ be an admissible set of relations and $[l^0]$ a tableau satisfying $\mathcal C$.  
%
The quasi typical $\g$-module $V_{\mathcal C}([l^0])$  is irreducible if and only if 
$\mathcal C$ is the
maximal set of relations satisfied by $[l^0]$
and $l^0_{m+n,i}\neq l^0_{m+n, j}, 1\leq i\leq m<  j\leq m+n$.
\end{theorem}

\begin{proof}
Suppose $\mathcal{C}$ is the maximal set of relations satisfied by $[l^0]$
and $l^0_{m+n,i}\neq l^0_{m+n, j}, 1\leq i\leq m<  j\leq m+n$. First we show that $V_{\mathcal C}([l^0])$ is generated by $[l^0]$. 
For any tableau  $ [l]\in V_{\mathcal{C}}([l^0])$ there exist $\{(i_t,j_t,\epsilon_t)\}$, $\epsilon_t=\pm 1$, $1\leq t \leq s$, 
such that for any $r\leq s$, $[l^0+\sum_{t=1}^{r}\epsilon_t\delta_{i_t,j_t}]$ satisfies  
$\mathcal{C}$
and $[l]=[l^0+\sum_{t=1}^{s}\epsilon_t \delta_{i_t,j_t}]$.

\begin{itemize}
\item [(a)]
Suppose $[l^0]$ and $[l]=[l^0+\delta_{ki}]$  satisfy $\mathcal{C}$.
\begin{itemize}
\item [(i)]
If $1\leq i\leq k\leq m-1$, then $l^0_{k+1,j}\neq l^0_{ki}$ for any $1\leq j\leq k+1$. Therefore the coefficient of $[l]$ in
$E_{k,k+1 }[l^0]$ is nonzero by the maximality of $\mathcal{C}$. Hence 
$[l] \in \g[l^0]$ by  Corollary \ref{corollary separation}. 

\item [(ii)]
If $1\leq i\leq m, k=m$, then $\theta_{mi}(l^0)=1$, i.e. $l^0_{m+1,i}=l^0_{m,i}$. Since $[l]$ satisfies $\mathcal C$, we have $l_{mj}\neq l_{mi}$ for any
$1\leq j\leq m$. It implies that 
$l^0_{mj}- l^0_{mi}-1\neq 0$ for any
$1\leq j\leq m$.
Therefore, the coefficient of $[l]$ in
$E_{k,k+1 }[l^0]$ is nonzero. By Corollary \ref{corollary separation},
$[l] \in \g[l^0]$. 

\item [(iii)]
If $1\leq i\leq m, m+1\leq  k\leq m+n-1$,
then $\theta_{ki}(l^0)=1, \theta_{k-1,i}(l^0)=0$.
Since $[l]$ satisfies $\mathcal C$, $l^0_{kj}-l^0_{ki}-1\neq 0$ for any $j\neq i,1\leq j\leq m$. Therefore, the coefficient of $[l]$ in
$E_{k+1,k }[l^0]$ is nonzero. By Corollary \ref{corollary separation},
$[l] \in \g[l^0]$. 


\item [(iv)]
Suppose that $m+1\leq i\leq k\leq m+n-1$.
Since  $\mathcal{C}$ is the maximal set of relations satisfied by $[l^0]$, $l_{ki}\neq l_{k-1,j}$ for any $m+1\leq j\leq k-1$. Then the coefficient of $[l]$ in
$E_{k+1,k}[l^0]$ is nonzero. By Corollary \ref{corollary separation},
$[l] \in \g[l^0]$. 
\end{itemize}
\item[(b)]
Suppose $[l^0]$ and $[l]=[l^0-\delta_{ki}]$  satisfy $\mathcal{C}$.
If $1\leq i\leq m<   k\leq m+n-1$,
then $\theta_{ki}(l^0)=0$.
Since $[l]$ satisfies $\mathcal C$, $l^0_{kj}-l^0_{ki}+1\neq 0$ for any $1\leq j\leq m$. 

In the following we show that
$$
\frac{\Pi_{j=m+1}^{k+1}(l_{k+1,j}-l_{ki})\Pi_{j=m+1}^{k-1}(l_{k-1,j}-l_{ki}+1)}
  {\Pi_{j=m+1}^{k} (l_{kj}-l_{ki})(l_{kj}-l_{ki}+1)}\neq 0
$$

If $l_{k+1,j}=l_{k,i}$ for some $j$, then there exist tableau $[l']$
such that ${l'}_{kq}={l}_{kq}+z_{kq}$, where  $z_{kq}\in \mathbb Z$, $z_{m+n,q}=0$ for all $q$,
$[l']$ satisfies $\mathcal C_1$ and $\mathcal C_2$,
$\l'_{kq}-\l'_{k-1,q}\equiv\theta_{k-1,q}(l)\in\{0,1\}, 1\leq q\leq m, m+1\leq k\leq m+n$,
$l_{k+1,j}=l_{k+1,i}$ which contradicts with the definition of tableau satisfying set of relations.
Therefore the coefficient of $[l]$ in
$E_{k+1,k}[l^0]$ is nonzero. By Corollary \ref{corollary separation},
$[l] \in \g[l^0]$. 
 
Using the same argument, it can be shown that 
$[l] \in \g[l^0]$ in all other cases.
\end{itemize}

Repeating the argument $s$ times we conclude that $V_{\mathcal{C}}([l^0])$ is generated by $[l^0]$. 
Suppose  $W$ is a nonzero submodule of $V_{\mathcal{C}}([l^0])$. Then $W$ contains a tableau $[l]$ by Corollary \ref{corollary separation}.
Using the above argument one can show  that $[l]$ generates $V_{\mathcal{C}}([l])$ which is in fact equal to $V_{\mathcal{C}}([l^0])$.
This proves the irreducibility of $V_{\mathcal{C}}([l^0])$.

Conversely, suppose
$\mathcal{C}$ is not the maximal set of relations satisfied by $[l^0]$. Then
$l_{k+1, i}^0-l_{k, j}^0\in \mathbb{Z}$ for some indexes and there is no relation between $(k+1, i)$ and $(k, j)$.
So there exist tableaux $[l^1], [l^2]\in  V_{\mathcal C}([l^0])$ such that $l_{k+1, i}^1-l^1_{k, j}\in \mathbb{Z}_{\geq 0}$ and $l^2_{k , j}-l^2_{k+1, i}\in \mathbb{Z}_{>0}$.
By  the Gelfand-Tsetlin formulas one has that  $[l^2]$ is not in the submodule  of $V_{\mathcal{C}}[l^0])$ generated by $[l^1]$ and thus $V_{\mathcal{C}}([l^0])$ is not irreducible.

Suppose that $l_{m+n,i}=l_{m+n, j}$ for some $1\leq i\leq m<  j\leq m+n$. Then there exist tableaux $[l^1],[l^2]\in V_{\mathcal{C}}[l^0])$ such that
$\theta_{m+n-1,i}(l^1)=0$, i.e. $l^1_{m+n-1,i}=l^1_{m+n, j}$ and  $\theta_{m+n-1,i}(l^2)=1$. 
By  the Gelfand-Tsetlin formulas one has that  $[l^2]$ is not in the submodule  of $V_{\mathcal{C}}[l^0])$ generated by $[l^1]$ and thus $V_{\mathcal{C}}([l^0])$ is not irreducible.
\end{proof}
\begin{theorem}\label{Kac-module}
$V_{\mathcal{C}}([l^0])\simeq K(V_{\mathcal{C}}([l^0])^{\g_1})$
if $V_{\mathcal{C}}([l^0])$ is irreducible.
\end{theorem}

\begin{proof}
It follows immediately from the Gelfand-Tsetlin formulas that
\begin{equation}
W=\{[l]\in V_{\mathcal{C}}([l^0])| \
\theta_{ki}(l)=0, 1\leq i\leq m\leq k\leq m+n-1\}\subset
V_{\mathcal{C}}([l^0])^{\g_1}.
\end{equation}
Suppose there exists $v\notin W$ such that $v\in V_{\mathcal{C}}([l^0])^{\g_1}$. Since $\Gamma$ separates the tableaux by Proposition \ref{prop action of gt subalg}, we may assume that  
$v=[l^1]$. 
For any tableau $[l]\in W$ we have $(E_{11}+\cdots+E_{mm})[l]=(l_{m1}+\cdots+l_{mm})[l]$. Then
 $(E_{11}+\cdots+E_{mm})[l^1]=(l_{m1}+\cdots+l_{mm}-\alpha)[l^1]$
for some positive integer $\alpha$.
Since the elements of $U(\g_0)$ commute with $E_{11}+\cdots+E_{mm}$, then 
 all tableaux in  $U(\g)[l^1]$
have eigenvalues which  are smaller or equal  than the eigenvalue of  $[l^1]$.
 This contradicts  the irreducibility of $V_{\mathcal{C}}([l^0])$. Thus 
$W=V_{\mathcal{C}}([l^0])^{\g_1}$
and  $K(V_{\mathcal{C}}([l^0])^{\g_1})=K(W)\simeq U(\g_{-1})\otimes W$.

Define the linear map $\pi:K(W)\longrightarrow V_{\mathcal{C}}([l^0])$
by
\begin{equation}
\pi(x\otimes [l])=x[l],
\end{equation}
$x\in U(\g_{-1})$ and $[l]\in W$.
It is easy to see that $\pi$ is a $\g$-module homomorphism.
Since the image of $\pi$ is nonzero and $V_{\mathcal{C}}([l^0])$
is irreducible, then $\pi$ is surjective.

The algebra $U(\g_{-1})$ has a basis
$$ E_{m+1,1}^{\epsilon_1}E_{m+1,2}^{\epsilon_2}\cdots
E_{m+1,m}^{\epsilon_m}E_{m+2,1}^{\epsilon_{m+1}}
\cdots
E_{m+n,m}^{\epsilon_{mn}}, \quad \epsilon_{i}=0 \text{ or }1, 1\leq i\leq mn.$$
Suppose $\pi(v)=0$. Then $v$ can be written as
$$v=\sum_{i}x_i\otimes w_i,$$
where $x_i$ is a product of $E_{kj}$, $1\leq j\leq m< k\leq m+n$ 
with $x_i\neq x_j$ for $i\neq j$ and $w_i$ is  a nonzero vector in $W$.
Without loss of generality we assume that
$x_1$ has the smallest number of factors.
Let $y$ be the product of  all negative odd root elements $E_{kj}$'s which are not factors of $x_1$. Then we have 
$\pi(yv)=0$ and $yx=z\otimes w_1$,
where $z$ is the product of all negative odd root elements.
Thus $z w_1=0$ in  $V_{\mathcal{C}}([l^0])$.

The module $V_{\mathcal{C}}([l^0])$ is irreducible, thus
$\mathcal C=(\mathcal C_1,\mathcal C_2)$ is a maximal set of relations satisfied by $[l^0]$ and $W=U(\g_0)w_1 $ is irreducible $\g_0$-module.
Now we will show that $zw_1\neq 0$ which gives a contradiction.
Let  $[l']$ be the tableau in $V_{\mathcal{C}}([l^0])$ such that
$$l'_{ki}=
l_{ki}-(m+n-k), 1\leq i\leq m\leq k\leq m+n.$$
Then there exists an element $u\in U(\g)$
such that $uw_1=[l']$. The element $u$ can be written in the following form:
\begin{equation}
u=\sum_{i} u^{(i)}_{1}u^{(i)}_{0}u^{(i)}_{-1},
\end{equation}
where $u^{(i)}_{\pm 1}\in U(\g_{\pm 1})$, $u^{(i)}_{0}\in U(\g_{ 0})$. 
By the  formula  \eqref{action h} we have $(E_{11}+\cdots+E_{mm})[l']=(l_{m1}+\cdots+l_{mm}-mn)[l']$.
Since $u^{(i)}_{0}$ does not change the eigenvalue of the tableau and $[l']$ has the smallest eigenvalue, then $u^{(i)}_{1}=a_i$ and 
$u^{(i)}_{-1}=b_iz$, $a_i,b_i\in \mathbb C$ for any $i$.
Thus $u$ has the form $u_0 z, u_0\in U(\g_0)$.
Therefore, $[l']=uw_1=u_0zw_1$ which implies that $zw_1$ is nonzero vector in $V_{\mathcal{C}}([l^0])$ which is a contradiction.
We conclude that $\pi$ is injective and
$V_{\mathcal{C}}([l^0])\simeq K(V_{\mathcal{C}}([l^0])^{\g_1})$.

%
%
%
%
%
%
%
%
%
%

\end{proof}

\

\section{Quasi covariant modules}
In this section we extend our construction of quasi typical modules to include all covariant tensor modules constructed in \cite{Stoilova2010}.
We recall their construction.

\begin{theorem}\cite{Stoilova2010}
 The set of vectors $\xi_{\La}$ parameterized by all the Gelfand-Tsetlin  tableaux $\La$ satisfying the conditions
\begin{enumerate}
  \item $\la_{m+n,i}\in \mathrm{Z}_{\geq 0}$  \text{ are fixed and } $\la_{m+n,j}-\la_{m+n,j+1}\in \mathrm{Z}_{\geq 0}, j\neq m, 1\leq j\leq m+n-1$, $\la_{m+n,m}\geq\#\{i:\la_{m+n,i}>0,\ m+1\leq i\leq m+n\};$
  \item  $\la_{ki}-\la_{k-1,i}\equiv\theta_{k-1,i}\in\{0,1\}, 1\leq i\leq m;m+1\leq k\leq m+n$;

  \item $\la_{km}\geq\#\{i:\la_{ki}>0,\ m+1\leq i\leq k\}, m+1\leq k\leq m+n$;
  \item if $\la_{m+1,m}=0$, {\it then} $\theta_{mm}=0$;
  \item $\la_{ki}-\la_{k,i+1}\in \mathrm{Z}_{\geq 0}, 1\leq i\leq m-1;m+1\leq k\leq m+n-1$;

  \item  $\la_{k+1,i}-\la_{ki}\in \mathrm{Z}_{\geq 0}$ {\it and} $\la_{k,i}-\la_{k+1,i+1}\in \mathrm{Z}_{\geq 0},$
$1\leq i\leq k\leq m-1$ {\it or} $m+1\leq i\leq k\leq m+n-1$,
\end{enumerate} constitutes a basis of $L(\la)$, where $\la=(\la_{1},\la_{2},\ \ldots,\la_{m+n})$ and $\la_{i}=\la_{m+n,i}$, $i=1, \ldots, m+n$.
The action of generators of $\gl({m|n})$ on the irreducible covariant tensor module $L(\la)$ is given by
the formulas \eqref{action h}-\eqref{action f}.
\end{theorem}

Let $\mathcal C=(\mathcal{C}_1,\mathcal{C}_2)$ be an admissible set of relations. We will say  that $\mathcal C$ is a \emph{covariant admissible} set of relations if it satisfies the following conditions:
\begin{enumerate}
\item $\mathcal{C}_2$ has a subset 
$$\mathcal C'=\{((i,i);(i+1,i+1))|p\leq i\leq m+n-1\},$$
for a fixed $p$ with $m+1\leq p\leq m+n$
and there is no $(p-1,i)$ such that $((p,p);(p-1,i))\in \mathcal C_2$ or $((p-1,i);(p,p))\in \mathcal C_2)$.
\item If $(k,i)$ is in the component of $\mathcal C'$ then
 $(k,j)$ is in the same component for $i\leq j\leq k$. 
\end{enumerate}

  
\begin{definition}
A tableau $[l]$ of $\gl(m|n)$ is  called \emph{$\mathcal{C}$-covariant} if:
\begin{enumerate}
\item $[l]$ satisfies $\mathcal C_1$ and $\mathcal C_2$;
\item  $\l_{ki}-\l_{k-1,i}\equiv\theta_{k-1,i}(l)\in\{0,1\}, 1\leq i\leq m, m+1\leq k\leq m+n$;
\item  $l_{m+n,n}-l_{m+n,m+n}\in \mathbb Z$, and $\l_{ki}- \l_{kj}\notin \mathbb Z$ for $m+1\leq k\leq m+n$, $1\leq i\leq m<j \leq k$ with 
$(k,j)$ not in the component of $\mathcal C'$.

  \item If $l_{m+n,n}-l_{m+n,m+n}\leq 0$ then for $m+1\leq k\leq m+n$
\begin{itemize}  
\item $l_{km}\geq l_{m+n,m+n}-m-n+p_k$, if there exists $p_k$ such that $p_k$ is the  smallest number among $\{m+1, m+2,\ldots, k\}$ 
with $l_{k,p_k}-l_{m+n,m+n}=p_k-m-n$.
\item $l_{km}\geq l_{m+n,m+n}-m-n+k+1$ if there exists no such $p_k$.
\end{itemize}
   \item If $l_{m+1,m}-l_{m+n,m+n}=1-n$ then $\theta_{mm}=0$.

\end{enumerate}
\end{definition}

Let 
$[l^0]=(l^0_{ij})$ be a $\mathcal C$-covariant tableau and ${\mathcal B}_{\mathcal{C}}([l^0])$ the set of all possible  $\mathcal C$-covariant tableaux $[l]=(l_{ij})$
satisfying 
$${l}_{ki}=l^0_{ki}+z_{ki},  \,\, z_{ki}\in \mathbb Z, \,\, 1\leq k\leq m+n-1, \,\, 1\leq i\leq k.$$

Denote by $V_{\mathcal C}([l^0])$  the vector space spanned by ${\mathcal B}_{\mathcal{C}}([l^0])$.
We have the following extension of Theorem \ref{admissible} and Theorem \ref{irreducibility} for  covariant admissible set of relations.

\

\begin{theorem}\label{quasi covariant}
Let $\mathcal{C}$ be a covariant admissible set of relations. 
\begin{itemize}
\item[(i)] The
the formulas \eqref{action h}-\eqref{action f} define a $\gl(m|n)$-module structure on $V_{\mathcal C}([l^0])$ for any $\mathcal{C}$-covariant tableau
$[l^0]$. 
\item[(ii)]
The action of $B(t)$ on any basis tableau $[l]$ of  $V_{\mathcal C}([l^0])$ is given by
\begin{equation}\label{action of gt subalg}
B_k(t)[l]=
\left\{ 
\begin{aligned}
&(1+tl_{k1})\cdots(1+tl_{kk})[l],      &\text{if } 1\leq k\leq m,\\
&\frac{(1+tl_{k1})\cdots(1+tl_{km})}
{(1+tl_{k,m+1})\cdots(1+tl_{k,k})} [l],    &\text{if } m+1\leq k\leq m+n.\\
\end{aligned}
 \right.
\end{equation}
\item[(iii)] The module
$V_{\mathcal C}([l^0])$ is irreducible if and only if
$\mathcal C$ is the
maximal set of relations satisfied by $[l^0]$.
\end{itemize}
\end{theorem}

\

The proof of this theorem is fully analogous to the proof of Theorem \ref{admissible} and Theorem \ref{irreducibility} so we leave the details out.   The module $V_{\mathcal C}([l^0])$ in Theorem \ref{quasi covariant} will be called \emph{quasi covariant} $\g$-module.
Note that all covariant tensor finite dimensional modules are quasi covariant. Hence, we have a generalization of the class of finite dimensional covariant tensor modules.

\
%

%

\section*{Acknowledgements}
V.F. is supported  by  CNPq  (200783/2018-1) and by  Fapesp (2018/23690-6).
J. Z. was supported by  Fapesp  (2015/05927-0). V.S. was supported by NSF grant 1701532.
J. Z. acknowledges the hospitality and support of Max Planck Institute for
Mathematics in Bonn. All authors are grateful to  the International Center for Mathematics of SUSTech for support and hospitality during their visit when the paper was completed.

\bibliographystyle{mrl}

\end{document}